\documentclass[a4paper,11pt,reqno]{amsart}

\usepackage[latin1]{inputenc}

\usepackage{
amssymb,
tikz,
xcolor,
subfig
}
\usepackage[hidelinks]{hyperref}

\usepackage{geometry}
\DeclareGraphicsExtensions{.png, .pdf}

\newtheorem{thm}{Theorem}[section]
\newtheorem{lemma}[thm]{Lemma}
\newtheorem{prop}[thm]{Proposition}

\theoremstyle{remark}
\newtheorem{remark}{Remark}[section]

\theoremstyle{definition}
\newtheorem{defi}{Definition}[section]

\newcommand{\R}{\mathbb{R}}
\newcommand{\C}{\mathbb{C}}

\newcommand{\G}{\mathcal{G}}
\newcommand{\D}{\mathcal{D}}
\newcommand{\Q}{\mathcal{Q}}
\newcommand{\K}{\mathcal{K}}
\newcommand{\M}{\mathcal{M}}
\renewcommand{\L}{\mathcal{L}}
\newcommand{\F}{\mathcal{F}}
\newcommand{\h}{\mathcal{R}}
\newcommand{\cN}{{\mathcal{N}}}
\newcommand{\cU}{{\mathcal{U}}}
\newcommand{\E}{\mathrm{E}}

\newcommand{\re}{\mathrm{Re}}

\newcommand{\sech}{\mathrm{sech}}

\renewcommand{\v}{\mathrm{v}}
\newcommand{\I}{{\mathbb{I}}}

\newcommand{\ep}{\varepsilon}
\newcommand{\om}{\omega}
\newcommand{\la}{\Lambda}
\newcommand{\ga}{\Gamma}

\newcommand{\dx}{\,dx}
\newcommand{\dy}{\,dy}
\newcommand{\dxe}{\,dx_e}

\newcommand{\dtau}{\,d\tau}

\DeclareMathOperator*{\esssup}{ess\,sup}
\newcommand{\lap}{\Delta_\G}
\newcommand{\lapd}{\widetilde{\Delta}_\G}
\newcommand{\Dg}{\D_\G}
\newcommand{\Dgc}{\D_{\G,c}}
\newcommand{\Dec}{\D_{e,c}}
\newcommand{\Dm}{\widetilde{\D}_{\G,c}}
\newcommand{\Dme}{\widetilde{\D}_{e,c}}
\newcommand{\dom}{\mathrm{dom}}
\newcommand{\dv}[2]{\frac{d#1}{d#2}}

\renewcommand{\geq}{\geqslant}
\renewcommand{\leq}{\leqslant}
\newcommand{\f}[2]{\frac{#1}{#2}}
\newcommand{\tf}[2]{\tfrac{#1}{#2}}

\newcommand{\be}{\begin{equation}}
\newcommand{\ee}{\end{equation}}

\tikzstyle{infinito}=[circle,inner sep=0pt,minimum size=0mm]
\tikzstyle{nodo}=[circle,draw,fill,inner sep=0pt, minimum size=0.5*width("k")]
\tikzstyle{nodo_vuoto}=[circle,draw,inner sep=0pt, minimum size=0.5*width("k")]
\usetikzlibrary{graphs}
\tikzset{every loop/.style={min distance=10mm,in=300,out=240,looseness=10}}
\tikzset{place/.style={circle,thick,draw=blue!75,fill=blue!20,minimum
size=6mm}}
\tikzset{place2/.style={circle,thick,draw=red!75,fill=red!20,minimum
size=6mm}}

\title[On the NLDE on noncompact metric graphs]{On the Nonlinear Dirac equation on noncompact metric graphs}

\author[W. Borrelli]{William Borrelli}
\address[W. Borrelli]{Scuola Normale Superiore, Centro De Giorgi, Piazza dei Cavalieri 3, I-56100, Pisa, Italy.} 
\email{william.borrelli@sns.it}

\author[R. Carlone]{Raffaele Carlone}
\address[R. Carlone]{Universit\`{a} ``Federico II'' di Napoli, Dipartimento di Matematica e Applicazioni ``R. Caccioppoli'', MSA, via Cinthia, I-80126, Napoli, Italy.} 
\email{raffaele.carlone@unina.it}

\author[L. Tentarelli]{Lorenzo Tentarelli}
\address[L. Tentarelli]{Politecnico di Torino, Dipartimento di Scienze Matematiche ``G.L. Lagrange'', Corso Duca degli Abruzzi 24, 10129, Torino, Italy.} 
\email{lorenzo.tentarelli@polito.it}

\begin{document}


\begin{abstract}
The paper discusses the Nonlinear Dirac Equation with Kerr-type nonlinearity (i.e., $|\psi|^{p-2}\psi$) on noncompact metric graphs with a finite number of edges, in the case of Kirchhoff-type vertex conditions. Precisely, we prove local well-posedness for the associated Cauchy problem in the operator domain and, for infinite $N$-star graphs, the existence of standing waves bifurcating from the trivial solution at $\omega=mc^2$, for any $p>2$. In the Appendix we also discuss the nonrelativistic limit of the Dirac-Kirchhoff operator.
\end{abstract}


\maketitle

\vspace{-.5cm}
    \smallskip



\section{Introduction and main results}

Following the discussion initiated in \cite{BCT-SIMA}, in this paper we investigate the NonLinear Dirac Equation (NLDE) on noncompact metric graphs with a finite number of edges. In particular, we address both the local well-posedness of the time-dependent problem in the operator domain and the existence of the standing waves. Precisely, while the former problem is studied under very general assumptions, the result concerning the standing waves is limited to the so-called infinite $N$-star graphs. On the other hand, both the questions are discussed for any superquadratic power nonlinearity.

Before presenting the main results of the paper, however, some basics on metric graphs and some comments on the existing literature are in order.


\subsection{Basics on metric graphs}

Metric graphs (see, e.g., Figure \ref{fig-metric}) are connected  \emph{multigraphs} $\G=(\mathrm{V},\mathrm{E})$ endowed with a parametrization that associates each bounded edge $e\in\mathrm{E}$ with a closed and bounded interval of the real line $I_e=[0,\ell_e]$, and each unbounded edge $e\in\mathrm{E}$ with a copy of the half-line $I_e=\R^+$ (for details see, e.g., \cite{AST-CVPDE,BK}). As a consequence, they are locally compact metric spaces, the metric being given by the shortest distance along the edges. The orientation of the variable $x_e$ of each edge is free on the bounded edges, while on the unbounded edges, we set $x_e=+\infty$ for the vertices at infinity.

\begin{figure}[h]
\centering
{\includegraphics[width=.5\columnwidth]{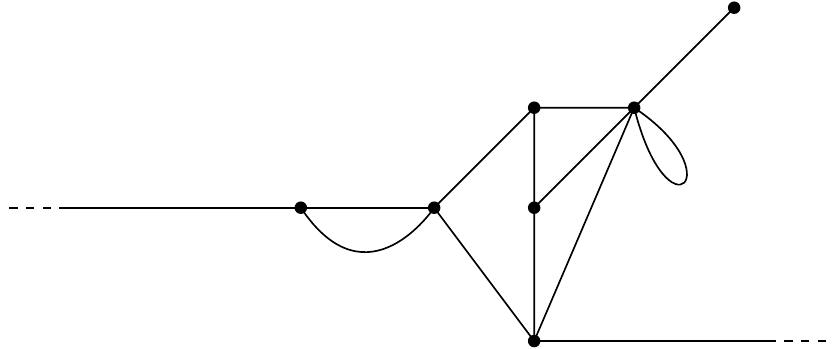}}
\caption{An example of noncompact metric graph.}
\label{fig-metric}
\end{figure}
\begin{remark}
Other choices are of course possibile for unbounded edges. Namely, one can identify (some of) them with the half-line $(-\infty,0)$ setting $x_e=-\infty$ for the vertex at infinity. Those edges can be considered then as `ingoing' half-lines. The choice made in the present paper is to consider only `outgoing' half-lines, in order to simplify notations.
\end{remark}

In addition, metric graphs are said to be \emph{compact} if they consist of a finite number of bounded edges and \emph{noncompact} otherwise. In particular, noncompact graphs consisting only of a finite number of unbounded edges incident at the same vertex are called \emph{infinite} $N$\emph{-star graphs} (see, e.g., Figure \ref{fig-nstar}).

\begin{figure}[h]
\centering
{\includegraphics[width=.35\columnwidth]{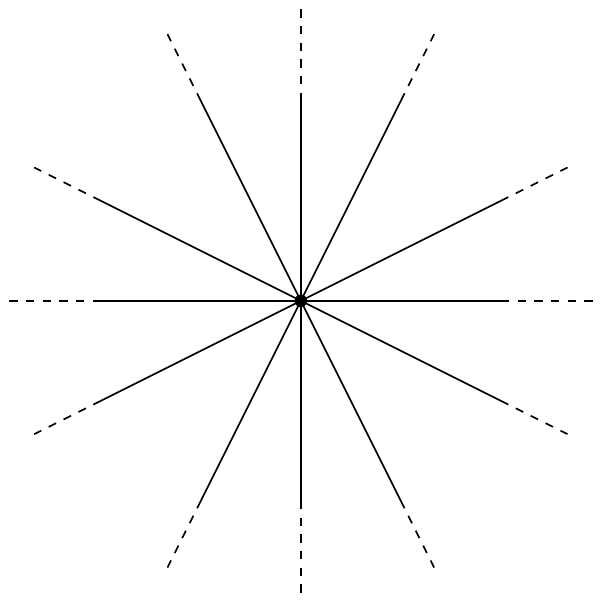}}
\caption{Example of infinite $N$-star graph ($N=12$).}
\label{fig-nstar}
\end{figure}

In view of their metric structure, one can define functions on metric graphs, i.e. $u:\G\to\C$, as families of functions $(u_e)_{e\in\mathrm{E}}$ defined on each edge, i.e. $u_e:I_e\to\C$, in such a way that $u_{|_e}=u_e$. Consistently, 
\[
 L^p(\G):=\bigoplus_{e\in\mathrm{E}}L^p(I_e),\qquad p\in[1,\infty],
\]
and
\[
 H^m(\G):=\bigoplus_{e\in\mathrm{E}}H^m(I_e),\qquad m\in[1,\infty],
\]
both endowed with the natural norms and, in the latter case, scalar products. Analogously, one can define \emph{spinors} on graphs, i.e. $\psi=(\phi,\chi)^T:\G\to\C^2$, as a collection of 2-spinors $(\psi_e)_{e\in\mathrm{E}}$ defined on each edge, i.e. $\psi_e=(\phi_e,\chi_e)^T:I_e\to\C^2$, so that Lebesgue and Sobolev spaces are again given by
\[
 L^p(\G,\C^2):=\bigoplus_{e\in\mathrm{E}}L^p(I_e,\C^2),\qquad p\in[1,\infty],
\]
and
\[
 H^m(\G,\C^2):=\bigoplus_{e\in\mathrm{E}}H^m(I_e,\C^2),\qquad m\in[1,\infty],
\]
endowed with the natural norms. 


\subsection{Background and known results}
\label{subsec:back}

Metric graphs gained a certain popularity in the last years as they are supposed to be an effective model for the study of physical systems confined in branched thin structures. The main applications of these models concern Bose-Einstein Condensates (BEC) and nonlinear optical fibers \cite{GW-PRE, HC-PD}. More precisely, the main features of these physical systems are well described by the \emph{focusing} NonLinear Schr\"odinger Equation (NLSE), i.e.
\begin{equation}
 \label{eq-NLStime}
 \imath\f{\partial w}{\partial t}=-\lap w-|w|^{p-2}\,w,\qquad p\geq2,
\end{equation}
where $-\lap$ is a suitable self-adjoint realizations of the standard Laplacian on a graph (see, e.g., \cite{AST-PARMA,GS-AP,GW-PRE,LLMOSCMWCCT-PLA} and references therein).

Most of the studies on this equation have focused on the \emph{standing-waves}, i.e. $L^2$-solutions of \eqref{eq-NLStime} of the form $w(t,x):=e^{-\imath\lambda t}u(x)$, in the case where $-\lap$ represents the self-adjoint operator usually called \emph{Kirchhoff} Laplacian, which is defined as follows.

\begin{defi}[Kirchhoff Laplacian]
\label{defi-lapk}
 The Laplacian with \emph{Kirchhoff vertex conditions}, denoted by $-\lap$, is the operator with action
 \[
  -\lap u_{|e}:=-u_e'',\qquad\forall e\in\E,
 \]
 (where $'$ denotes the derivative with respect the $x_e$ variable) and domain
 \[
 \dom(-\lap):=\big\{u\in H^2(\G):\text{\eqref{eq-cont} and \eqref{eq-kirch} are satisfied}\big\},
\]
with
\begin{gather}
 \label{eq-cont} u_e(\v)=u_f(\v),\qquad\forall e,f\succ\v,\qquad\forall\v\in \mathrm{V}_\K,\\[.2cm]
 \label{eq-kirch} \sum_{e\succ\v}\dv{u_e}{x_e}(\v)=0,\qquad\forall\v\in \mathrm{V}_\K,
\end{gather}
$e\succ\v$ meaning that $e$ is incident at $\v$, and $\mathrm{V}_\K$ denotes the set of vertices on the \emph{compact core} $\K$ of the graph (i.e., the subgraph of the bounded edges) and $\f{du_e}{dx_e}(\v)$ representing $u_e'(0)$ or $-u_e'(\ell_e)$ depending whether $\v$ is the initial or terminal vertex of the edge.
\end{defi}

 Looking for standing waves of \eqref{eq-NLStime} is equivalent to finding functions $u$ such that
\begin{equation}
\label{eq-NLS}
\left\{
\begin{array}{l}
 \displaystyle u\in\dom(-\lap),\,u\neq 0,\\[.2cm]
 \displaystyle -u_e''-|u_e|^{p-2}u_e=\lambda u_e,\qquad\forall e\in\E,
 \end{array}
 \right.
\end{equation}
which are also called \emph{bound states}. Results on the existence and multiplicity or nonexistence of the bound states for noncompact graphs with a finite number of edges can be found, for instance, in \cite{AST-CVPDE,AST-JFA,AST-CMP,AST-CVPDE2,DST-ADV,KP-JDE,NPS-Non,NRS-JDE}, while for the investigation of periodic graphs (e.g., Figure \ref{fig-period}) and compact graphs (e.g., Figure \ref{fig-metric} without the two half-lines) we mention, for instance, \cite{ADST-APDE,D-NODEA,GPS-NODEA,PS-AHP} and \cite{CDS-MJM,D-JDE}. In addition, other self-adjoint realization of the Laplacian on graphs have been discussed in the last years. We recall \cite{ACFN-RMP,GI-JPA,GW-GOL} for the study of time-dependent problems and \cite{ACFN-ANHIPC,CFN-Non} for the study of the standing waves. Finally, it is worth recalling that also some problems of control theory have been studied on graphs (see \cite{AD-p1,D-p}).

\begin{figure}[t]
\centering
\subfloat[][one-dimensional periodic graph.]
{\includegraphics[width=.65\columnwidth]{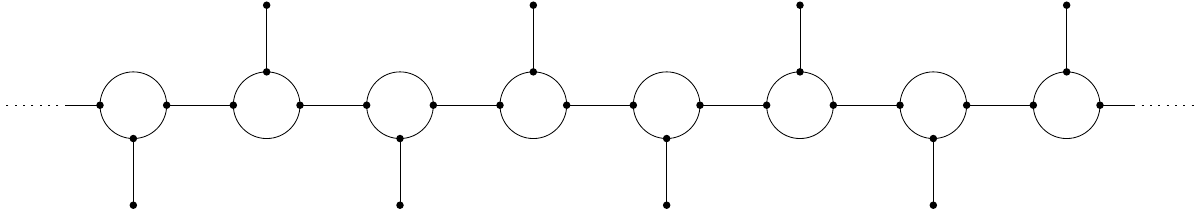}}
\\[.5cm]
\subfloat[][two-dimensional grid.]
{\includegraphics[width=.65\columnwidth]{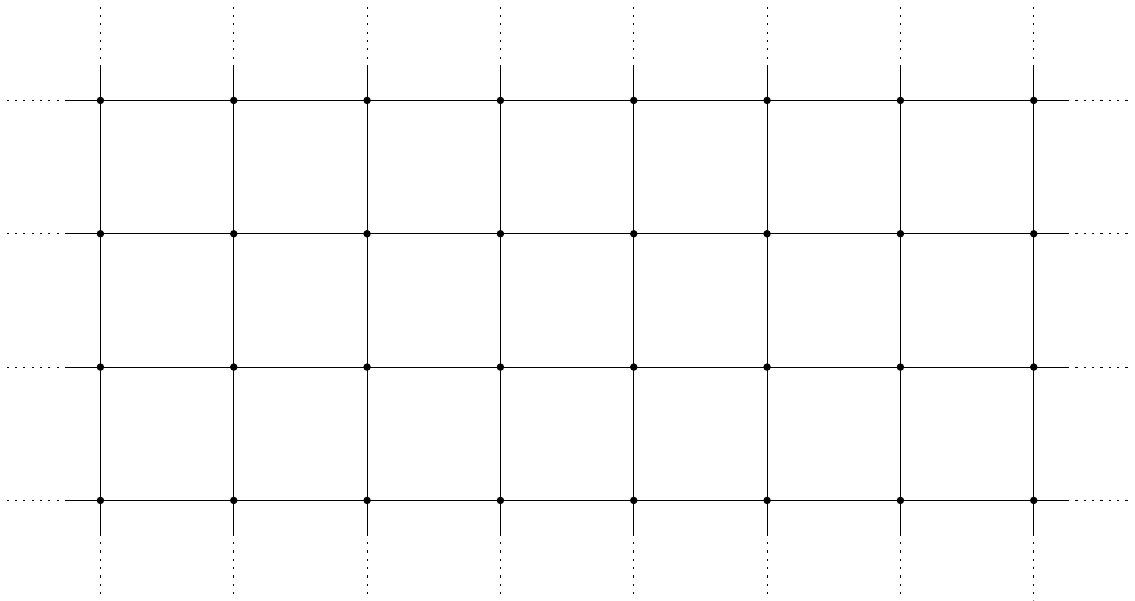}}
\caption{Two examples of periodic graphs.}
\label{fig-period}
\end{figure}

\medskip
Recently, the NLDE with Kerr-type nonlinearity
\[
 \imath\f{\partial \Psi}{\partial t}=\D\Psi-|\Psi|^{p-2}\Psi,\qquad p\geq2,
\]
where $\D$ is the $d$-dimensional Dirac operator, has attracted growing interest, mainly in the case $d=2$. Indeed, it has been proposed as an effective equation for the so-called \emph{Dirac materials} (such as, for instance, \emph{graphene} and \emph{germanene} \cite{WBB-AdP}) and for BEC in some particular regimes where relativistic effects cannot be neglected (see \cite{HC-PD}). The derivation of these effective models has been discussed in \cite{AS-JMP,FLW-APDE,FW-JAMS,FW-CMP}, while an exhaustive investigation of the standing waves can be found in \cite{B-JDE,B-JMP,B-CVPDE,BF-TAMS}.

In addition, some recent works proposed the NLDE as an effective equation for BEC in branched thin domains and optical fibers, in the presence of relativistic effects (see \cite{SBMK-JPA,TLB-AP}). The self-adjoint realization of the Dirac operator employed in these applications is the so-called \emph{Kirchhoff-type} one (for details see \cite[Appendix A]{BCT-SIMA} and \cite{BCT-p}), which is defined as follows:

\begin{defi}[Dirac-Kirchhoff]
\label{defi-dirac}
 The Dirac operator with \emph{Kirchhoff-type vertex conditions}, that we denote by $\Dg$, is an operator with action
\begin{equation}
 \label{eq-Dstand}
 \Dg\psi_{|_e}=\D_e\psi_e:=-\imath c\sigma_{1}\psi_e'+mc^{2}\sigma_{3}\psi_e,\qquad\forall e\in\E,
\end{equation}
$m>0$ and $c>0$ representing the \emph{mass} of the generic particle of the system and a \emph{relativistic parameter} (respectively), and $\sigma_1$ and $\sigma_3$ representing the \emph{Pauli matrices}
\begin{equation}
 \label{eq-pauli}
 \sigma_1:=\begin{pmatrix}
  0 & 1 \\
  1 & 0
 \end{pmatrix}
 \qquad\text{and}\qquad
 \sigma_3:=\begin{pmatrix}
  1 & 0 \\
  0 & -1
 \end{pmatrix},
\end{equation}
and domain
\[
 \dom(\Dg):=\{\psi=(\phi,\chi)^T\in H^1(\G):\text{\eqref{eq-cont_bis} and \eqref{eq-kirch_bis} are satisfied}\},
\]
where
\begin{gather}
 \label{eq-cont_bis} \phi_e(\v)=\phi_f(\v),\qquad\forall e,f\succ\v,\qquad\forall\v\in\mathrm{V}_\K,\\[.2cm]
 \label{eq-kirch_bis} \sum_{e\succ\v}\chi_e^{\pm}(\v)=0,\qquad\forall\v\in\mathrm{V}_\K,
\end{gather}
$\chi_e^{\pm}(\v)$ denoting $\chi_e(0)$ or $-\chi_e(\ell_e)$ depending on the orientation of the edge.
\end{defi}

\begin{remark}
 It is clear that the operator $\Dg$ does depend on both $m$ and $c$. However, in the following, this dependence will be omitted whenever it is not relevant. Also, we recall that the constant $c>0$ appearing in \eqref{eq-Dstand} usually represents the speed of light in relativistic theories. Anyway, since the model in the present paper is just an effective one, in this case it should be meant as a generic physical parameter that formally plays the same role. 
\end{remark}

\begin{remark}
 \label{rem-kirchname}
 The operator $\Dg$ is said Dirac-Kirchhoff operator since, if one considers the operator $\Dg^2$ on spinors consisting of the first component only, then one obtains precisely the Kirchhoff-Laplacian, plus zero-order terms. This is not the unique possible choice to find a selfadjoint realization of the Dirac operator on graphs; other choices have been studied, e.g., in \cite{BH-JPA,BT-JMP,H,P}.
\end{remark}
\begin{remark}
Notice that, since we are considering a first order differential operator, its definition depends on the orientation for the edges. For instance, in the case of two half-lines joined at the origin, a simple computation shows that the operator \eqref{defi-dirac} coincides with the standard Dirac operator on the real line, with domain $H^1(\R)$, only if one half-line is identified as $(-\infty,0)$ and the other with $(0,\infty)$.
\end{remark}

From the mathematical point of view, the sole study of the NLDE on graphs, i.e.
\begin{equation}
 \label{eq-NLDtime}
 \imath\f{\partial \Psi}{\partial t}=\Dg\Psi-|\Psi|^{p-2}\Psi,\qquad p\geq2,
\end{equation}
has been presented in \cite{BCT-SIMA}. More precisely, the paper discusses a slightly modified model in which the nonlinear potential affects just the compact core of the graph, i.e.
\begin{equation}
 \label{eq-NLDtimeconc}
 \imath\f{\partial \Psi}{\partial t}=\Dg\Psi-\chi_{\K}|\Psi|^{p-2}\Psi,\qquad p\geq2.
\end{equation}
This idea of considering a localized nonlinearity was first introduced in \cite{GSD-PRA} and widely studied in \cite{DT-CVPDE,DT-OTAA,ST-JDE,ST-NA,T-JMAA} in the Schr\"odinger case, with a specific focus on existence/nonexistence and multiplicity of the bound states, which in this case read
\begin{equation}
\label{eq-NLSconc}
\left\{
\begin{array}{l}
 \displaystyle u\in\dom(-\lap),\,u\neq 0,\\[.2cm]
 \displaystyle -u_e''-\chi_{\K}|u_e|^{p-2}u_e=\lambda u_e,\qquad\forall e\in\E,
 \end{array}
 \right.
\end{equation}
(see also \cite{BCT-SYM} for an overview on the subject).

Also regarding \eqref{eq-NLDtimeconc}, the discussion developed in \cite{BCT-SIMA} concerns the bound states, here defined as the spinors $\psi$ such that
\begin{equation}
\label{eq-NLDconc}
\left\{
\begin{array}{l}
 \displaystyle \psi\in\dom(\Dg),\,\psi\neq 0,\\[.2cm]
 \displaystyle \D_e\psi_e-\chi_{\K}|\psi_e|^{p-2}\psi_e=\omega \psi_e,\qquad\forall e\in\E,
 \end{array}
 \right.
\end{equation}
(so that $\Psi(t,x)=e^{-\imath\om t}\psi(x)$ is an $L^2$-solution of \eqref{eq-NLDtimeconc}, namely a standing wave). In particular, \cite{BCT-SIMA} proves that for every $\omega\in(-mc^2,mc^2)$ there exist infinitely many distinct bound states and discusses their behavior in the so-called \emph{nonrelativistic limit} ``$c\to+\infty$''. More precisely, one shows that for any sequence $c_n\to+\infty$, there exist (at least) a sequence $\omega_n\to+\infty$ such that any associated sequence of bound states $(\psi_n)_n$ converges (up to subsequences) to a spinor $(u,0)^T$ where $u$ is a bound state for the NLSE with localized nonlinearity.

\begin{remark}
 Actually, the limit function $u$ solves a version of \eqref{eq-NLSconc} with a coefficient $2m$ in front of the nonlinearity. However, this is consistent with the fact that in the nonrelativistic limit the kinetic part of the hamiltonian of a particle is formally given by $\frac{-\lap}{2m}$, in place of $-\lap$.
\end{remark}

\begin{remark}
 Recall that, since the spectrum of $\Dg$ is given by
 \[
  \sigma(\Dg)=(-\infty,-mc^2]\cup[mc^2,+\infty)
 \]
 (as proved in \cite[Appendix A]{BCT-SIMA} and \cite{BCT-p}), then $(-mc^2,mc^2)=\R\backslash\sigma(\Dg)$.
\end{remark}


\subsection{Main results}

In this paper, in contrast to \cite{BCT-SIMA}, we aim at discussing the NLDE on graphs without the assumption of the localization of the nonlinearity.

 As for the Schr\"odinger case (see \cite{ACFN-RMP}), the first aspect to be investigated concerning \eqref{eq-NLDtime} is the well-posedness of the associated Cauchy problem. It provides, indeed, the actual consistency of \eqref{eq-NLDtime} as an effective model for study of the dynamics of physical systems (such as, for instance, BEC in branched thin domains and optical fibers in the presence of relativistic effects).

Precisely, we state local well-posedness of the Cauchy problem associated with \eqref{eq-NLDtime} in the operator domain.

\begin{thm}[Local well-posedness]
\label{thm:lwp}
Let $\G$ be a noncompact metric graph with a finite number of edges and let $p>2$. Then:
\begin{itemize}
 \item[(i)] for any $\Psi_0\in \dom(\Dg)$, there exists $T>0$ such that there exists a unique solution 
\be
\label{eq:solution}
\Psi\in C([0,T],\dom(\Dg))\cap C^1([0,T],L^2(\G,\C^2)),
\ee
of the Cauchy problem
\be
\label{eq:cauchy2}
 \left\{\begin{array}{l}
  \displaystyle \imath\f{\partial \Psi}{\partial t}=\Dg\Psi-|\Psi|^{p-2}\Psi \\[.4cm]
  \displaystyle \Psi(0,\cdot\,)=\Psi_0
 \end{array}\right.;
\ee

\item[(ii)] a blow-up alternative holds, namely, letting $T^*$ be the supremum of the times for which (i) holds, there results
\[
\limsup_{t\to T^*}\|\Psi(t,\cdot\,)\|_{\Dg}<\infty\quad\implies\quad T^*=+\infty,
\]
where $\|\,\cdot\,\|_{\Dg}$ denotes the graph norm associated with $\Dg$, i.e.
\[
 \|\psi\|_{\Dg}:=\|\psi\|_{L^2(\G,\C^2)}+\|\Dg\psi\|_{L^2(\G,\C^2)};
\]

\item[(iii)] the $L^2$-norm of the solution is preserved along the flow, i.e.
\[
\|\Psi(t,\cdot\,)\|_{L^2(\G,\C^2)}=\|\Psi_0(\,\cdot\,)\|_{L^2(\G,\C^2)},\qquad\forall t\in(0,T^*);
\]
\item[(iv)] the energy of the solution is preserved along the flow, i.e.
\begin{multline}
\label{eq:energy}
\qquad\qquad E\big(\Psi(t,\cdot\,)\big):=\f{1}{2}\int_\G\langle\Psi(t,x),\Dg\Psi(t,x)\rangle_{\C^2}\dx+\\[.1cm]
-\frac{1}{p}\int_\G|\Psi(t,x)|^p\dx=E\big(\Psi_0\big),\qquad\forall t\in(0,T^*)
\end{multline}
(where $\langle\,\cdot\,,\,\cdot\,\rangle_{\C^2}$ denotes the standard hermitian product of $\C^2$).
\end{itemize}
\end{thm}

\begin{remark}
 Note that the sign in front of the nonlinearity in \eqref{eq:cauchy2} is not relevant for the proof of Theorem \ref{thm:lwp}. The result can be proved as well with the opposite sign.
\end{remark}

 As usual in the study of nonlinear dispersive equations, once checked the consistency of the model, the most natural question concerns the existence or nonexistence of the stationary solutions of \eqref{eq-NLDtime}. In the context of metric graphs, lacking general results, it is common to study first graphs that present trivial compact core, that is infinite $N$-star graphs (see again \cite{ACFN-RMP}). Therefore, the second main result of our paper deals with the existence of bound states of the NLDE on an $N$-star graph, i.e., spinors $\psi$ such that
\begin{equation}
\label{eq:stat}
\left\{
\begin{array}{l}
 \displaystyle \psi\in\dom(\Dg),\,\psi\neq 0,\\[.2cm]
 \displaystyle \D_e\psi_e-|\psi_e|^{p-2}\psi_e=\omega\psi_e,\qquad\forall e=1,\dots,N
 \end{array}
 \right.
\end{equation}
(with a little abuse of notation we denote by $e$ both the edge itself and its index). In particular, we prove such bound states to be strictly related to those of the NLSE, precisely to 
\begin{equation}
 \label{eq-NLSaux}
 \left\{
 \begin{array}{l}
 \displaystyle u\in\dom(-\lap),\,u\neq 0,\\[.2cm]
 \displaystyle \f{u''_{e}}{2m}+|u_{e}|^{p-2}u_{e}=u_{e},\qquad\forall e=1,\ldots,N\,.
 \end{array}
 \right.
\end{equation}

\begin{thm}[Standing waves]
\label{thm:standingwaves}
Let $\G$ be an infinite $N$-star graph and let $p>2$. Then, there exists $\ep_0>0$ such that for every $\om\in(mc^2-\ep_0,mc^2)$ there is at least a spinor $\psi_\om\neq0$ that satisfies \eqref{eq:stat}.

More precisely, for every real valued solution $U$ to \eqref{eq-NLSaux}, there exist $\ep_0>0$ such that for all $\om\in(mc^2-\ep_0,mc^2)$ there exists $(u_\om,v_\om)^T\in\dom(\Dg)$, $u_\om,v_\om\neq0$, that satisfies
\begin{equation}
 \label{eq-systaux}
 \left\{
 \begin{array}{l}
  -\imath v_{\om,e}'+u_{\om,e}=\big(|u_{\om,e}|^2+(mc^2-\om)|v_{\om,e}|^2\big)^{\f{p-2}{2}}u_{\om,e}\\[.3cm]
  -\imath u_{\om,e}'-(mc^2+\om)v_{\om,e}=(mc^2-\om)\big(|u_{\om,e}|^2+(mc^2-\om)|v_{\om,e}|^2\big)^{\f{p-2}{2}}v_{\om,e}
 \end{array}
 \right.
\end{equation}
with
\[
(u_{\om,e},v_{\om,e})^T_{|\om=mc^2}=\big(U_e,-\imath\f{U_e'}{2m}\big),
\]
for every $e=1,\dots,N$, and such that for every $\om\in(mc^2-\ep_0,mc^2)$ the spinor $\psi_\om=(\phi_\om,\chi_\om)^T$ defined by
\begin{equation}
\label{eq:sol_scaling}
\begin{array}{l}
 \displaystyle \phi_{\om,e}:=(mc^2-\om)^{\f{1}{p-2}}\,u_{\om,e}\big(\sqrt{mc^2-\om}\,x_e\big),\qquad\forall e=1,\ldots,N,\\[.3cm]
 \displaystyle \chi_{\om,e}:=(mc^2-\om)^{\f{p}{2(p-2)}}\,v_{\om,e}\big(\sqrt{mc^2-\om}\,x_e\big),\qquad\forall e=1,\ldots,N,
\end{array}
\end{equation}
fulfills \eqref{eq:stat}.
\end{thm}

In other words, Theorem \ref{thm:standingwaves} claims two things. The former is that there exists at least one branch of nontrivial bound states of the NLDE which bifurcate from the trivial solution at the bifurcation point $\om=mc^2$. The latter is that one can construct a branch of bound states of the NLDE from any real-valued bound state of the NLSE, since each bound state of the NLSE gives rise to a family of solutions of \eqref{eq-systaux}, by the scaling \eqref{eq:sol_scaling}.

\begin{remark}
 Note that, if the number of half-lines is odd, then there exists only one real-valued bound state of the NLSE; while, if the number of half-lines is even, then there exist infinitely many real-valued bound states (see \cite{ACFN-EPL} and Section \ref{sec-nls}). As a consequence, in the first case, our procedure allows constructing only one branch of bound states of the NLDE, whereas in the second case the possible branches are infinitely many.
\end{remark}

\begin{remark}\label{rmk:negativeedge}
As we will show in Remark \ref{rmk:derivdefocusing}, the method adopted to prove Theorem \ref{thm:standingwaves} does not apply in the regime $\omega\to-mc^2$. Indeed, in that case the limit equation \eqref{eq-NLSaux} is replaced by the following \emph{defocusing} NLSE 
\begin{equation}
 \label{eq:defocusing}
\left\{
 \begin{array}{l}
 \displaystyle v\in\dom(-\lap),\,v\neq 0,\\[.2cm]
 \displaystyle \f{v''_{e}}{2m}-|v_{e}|^{p-2}v_{e}=v_{e},\qquad\forall e=1,\ldots,N\,.
 \end{array}
 \right.
\end{equation}
Then multiplying the equation above by $v_e$, edge by edge, and integrating by parts one easily sees that \eqref{eq:defocusing} admits the sole trivial solution.

\end{remark}

It is also worth highlighting that, as one can see from the very statement of Theorem \ref{thm:standingwaves}, in this case one has to use a completely different proof strategy with respect to \cite{BCT-SIMA}, where the problem of localized nonlinearity was addressed. More precisely, the results present in \cite{BCT-SIMA} relied on the fact that the compactness of the nonlinearity simplifies the compactness analysis of the Palais-Smale sequences of the action functional
\[
 \mathcal{L}(\psi):=E(\psi)-\f{\om}{2}\int_\G|\psi(x)|^2\dx.
\]
At the moment, it is not clear how to adapt the techniques used in \cite{BCT-SIMA} to problems with an extended nonlinearity on a graph, even in the Schr\"odinger case. Nevertheless, it is possible to adapt a method introduced by Stuart in \cite{S-SIAM} (see also \cite{BC,OUN-DIE}) based on a suitable scaling of the equation and on the use of the Implicit Function Theorem (see Section \ref{sec:standingwaves}). Unfortunately, this approach requires the restriction of the discussion to infinite $N$-star graph (see also Remark \ref{rem-restriction1}). The discussion of the existenxe/nonexistence of stationary solutions on more general graphs and of the stability properties of such solutions, which are the natural further steps of our research, will be addressed in a forthcoming paper.

\medskip
Finally,  as we made in \cite{BCT-SIMA}, we aim at giving a more rigorous justification (compared to Remark \ref{rem-kirchname}) for choosing the vertex conditions \eqref{eq-cont_bis} and \eqref{eq-kirch_bis}, that we call ``of Kirchhoff-type',' in the self-adjoint realization of the Dirac operator on graphs. Such a discussion follows the natural connection between Schr\"odinger models and Dirac models in the nonrelativistic limit. However, in contrast to \cite{BCT-SIMA}, here the point of view is in some sense more `foundational', as it focuses on the behavior of the operator itself when $c\to+\infty$.

Such a limit shows that, up to a suitable renormalization (due to the rest energy of relativistic particles), the operator $\Dg$ introduced in Definition \ref{defi-dirac} converges to $-\lap/2m$, in the norm-resolvent sense. In other words, the nonrelativistic counterpart of $\Dg$ is the Kirchhoff Laplacian, which explains both the name given to $\Dg$ and our choice of such realization of the Dirac operator.

As a different renormalization leads to a different type of graph Laplacian in the nonrelativistic limit, before stating the last result of the paper, we also recall the following definition (see e.g., \cite{Sc-Bo}).

\begin{defi}[Homogeneous $\delta'$ Laplacian]
\label{defi-lapdp}
 The Laplacian with \emph{homogeneous} $\delta'$ \emph{vertex conditions}, denoted by $-\lapd$, is an operator with the same action of $-\lap$ and domain
 \[
 \dom(-\lapd):=\big\{u\in H^2(\G):\text{\eqref{eq-contd} and \eqref{eq-kirchd} are satisfied}\big\},
\]
with
\begin{gather}
 \label{eq-contd}  u_e'(\v)=u_f'(\v),\qquad\forall e,f\succ\v,\qquad\forall\v\in\mathrm{V}_\K,\\[.2cm]
 \label{eq-kirchd} \sum_{e\succ\v}u_e^{\pm}(\v)=0,\qquad\forall\v\in\mathrm{V}_\K,
\end{gather}
 where $u_e'(\v)$ denotes the value of the derivative at $\v$ independently of the fact that $x_e=0$ or $x_e=\ell_e$ at $\v$.
\end{defi}
In order to state the next result it is also necessary to define a matrix operator that gathers the Kirchhoff Laplacian and the homogeneous $\delta'$ Laplacian.

\begin{defi}
We denote by $-\Delta$ the operator with domain
\[
 \dom(-\Delta):=\{\psi=(\phi,\chi)^T:\phi\in\dom(-\lap),\,\chi\in\dom(-\lapd)\},
\]
with action
\begin{equation}
 \label{eq-biglap}
 -\Delta\psi:=\begin{pmatrix} -\lap & 0 \\ 0 & -\lapd \end{pmatrix}\begin{pmatrix} \phi \\ \chi \end{pmatrix},
\end{equation}
where $-\lap$ and $-\lapd$ are the operators given by Definitions \ref{defi-lapk} and \ref{defi-lapdp}.
\end{defi}

\begin{prop}[Nonrelativistic limit of $\Dg$]
 \label{prop:nonrel}
 Let $\G$ be a noncompact metric graph with a finite number of edges. Then, for every $k\in\C\backslash\R$
 \begin{equation}
  \label{eq-lim1}
  (\Dgc-mc^2-k)^{-1}\longrightarrow P^+\bigg(\frac{-\Delta}{2m} -k\bigg)^{-1},\qquad\text{as}\quad c\to+\infty,
 \end{equation}
 in the $L^\infty$-operator norm, with $P^+$ the projector on the first component of the spinor, represented by the matrix $\begin{pmatrix} 1 & 0 \\ 0 & 0 \end{pmatrix}$. Moreover,
 \begin{equation}
  \label{eq-lim2}
  (\Dgc+mc^2-k)^{-1}\longrightarrow P^-\bigg(\frac{-\Delta}{2m} -k\bigg)^{-1},\qquad\text{as}\quad c\to+\infty,
 \end{equation}
 in the same norm, with $P^-$ the projector on the second component of the spinor, represented by the matrix $\begin{pmatrix} 0 & 0 \\ 0 & 1 \end{pmatrix}$.
\end{prop}

As we mentioned before, the two different renormalizations of $\Dg$ given by \eqref{eq-lim1} and \eqref{eq-lim2} (respectively) yield two different limits. In the first case, one obtains the expected Kirchhoff Laplacian, while in the second case one obtains a Laplacian with the so-called homogeneous $\delta'$ vertex condition (see, e.g., \cite{BD-LMP}).

On the other hand, it is immediate by \cite[Theorem VIII.21]{RS-I} that Proposition \ref{prop:nonrel} entails the strong $L^2$ convergence of the propagators, i.e.
\begin{gather*}
 e^{-\imath t(\Dgc-mc^2)}\longrightarrow e^{\imath t \big(P^+\frac{-\Delta}{2m}\big)},\qquad\forall t\in\R,\\
 e^{-\imath t(\Dgc+mc^2)}\longrightarrow e^{\imath t \big(P^-\frac{-\Delta}{2m}\big)},\qquad\forall t\in\R.
\end{gather*}


\subsection{Organization of the paper}

The paper is organized as follows:

\begin{itemize}
 \item[(i)] in Section \ref{sec:lwpandgwp} we prove Theorem \ref{thm:lwp};
 \item[(ii)] in Section \ref{sec:standingwaves} we give the proof of Theorem \ref{thm:standingwaves};
 \item[(iii)] the Appendix is devoted to the proof of Proposition \ref{prop:nonrel}.
\end{itemize}


\section{Proof of Theorem  \ref{thm:lwp}}
\label{sec:lwpandgwp}
In this section we prove Theorem \ref{thm:lwp}.

Recall that, using the anti-commutation properties of the Pauli matrices,
\[
 \int_\G|\Dg\psi|^2\dx=\int_{\G}|\psi'|^2\dx+m^2\int_\G|\psi|^2\dx,\qquad\forall\psi\in\dom(\Dg),
\]
so that
\[
 \|\psi\|_{H^1(\G,\C^2)}\sim\|\Dg\psi\|_{L^2(\G,\C^2)}\sim\|\psi\|_{\Dg},\qquad\forall\psi\in\dom(\Dg).
\]
We also note that, in the above formulas as well as in the next two sections, we will always assume $c=1$ since such parameter does not play any role in the proofs.

In addition, we will often use (with a little abuse) both $\Psi(t,\cdot\,)$ and $\Psi(t)$ to denote the function of space that one obtains freezing the time variable. On the other hand, the notation ``$x$'' in the following has to be meant as a generic point of the graph, without a precise specification of the edge one refers too. 


\subsection{Local well-posedness in the operator domain and blow-up alternative}

In order to prove Theorem \ref{thm:lwp} we first show the following lemma, which states the equivalence between \eqref{eq:cauchy2} and the associated \emph{Duhamel formula}. Note that in the following, we will always denote $e^{-\imath t\Dg}$ by $\cU(t)$.

\begin{lemma}[Duhamel formula]
\label{lem:duhamel}
 Let $\G$ be a noncompact metric graph with a finite number of edges and let $p>2$. For every $T>0$, a function $\Psi\in C([0,T],\dom(\Dg))\cap C^1([0,T],L^2(\G,\C^2))$ is a solution of \eqref{eq:cauchy2} if and only if it solves 
 \begin{equation}
  \label{eq:duhamel}
  \Psi(t,x)=\cU(t)\Psi_0[x]+\imath\int^t_0 \cU(t-\tau)\big(|\Psi(\tau)|^{p-2}\Psi(\tau)\big)[x]\dtau
 \end{equation}
 in $L^2(\G,\C^2)$, for every $t\in[0,T]$.
\end{lemma}

\begin{proof}
 It is sufficient to prove that \eqref{eq:duhamel} entails \eqref{eq:cauchy2}. The converse can be easily checked simply retracing backward the following computations.
 
 First we show that $\Psi(0)=\Psi_0$. Since $\cU(t)$ is a strongly continuous group on $L^2(\G,\C^2)$ it is left to prove that the integral term in \eqref{eq:duhamel} vanishes as $t\to0$. To this aim it suffices to see that, for every $T>0$,
 \begin{equation}
  \label{eq:nonlin1}
  \cU(T-\cdot\,)\big(|\Psi(\,\cdot\,)|^{p-2}\Psi(\,\cdot\,)\big)[\,\cdot\,]\in L^1([0,T],L^2(\G,\C^2))
 \end{equation}
 and exploit the absolute continuity of the Lebesgue integral. However, \eqref{eq:nonlin1} is a direct consequence of the fact that $\cU(t)$ is an isometry of $L^2(\G,\C^2)$, which yields
 \begin{multline*}
  \int_0^T\big\|\cU(T-\tau)(|\Psi(\tau)|^{p-2}\Psi(\tau))\big\|_{L^2(\G,\C^2)}\dtau=\int_0^T\big\||\Psi(\tau)|^{p-2}\Psi(\tau)\big\|_{L^2(\G,\C^2)}\dtau\\[.1cm]
  \leq T\|\Psi\|_{C([0,T],H^1(\G,\C^2))}^{p-1}<+\infty
 \end{multline*}
were we used $\Psi\in C([0,T],\dom(\Dg))$, $\dom(\Dg)\hookrightarrow H^1(\G,\C^2)$ and the Sobolev embedding.

 On the other hand, one has to prove that the function $\Psi$ that satisfies \eqref{eq:duhamel} also fulfills \eqref{eq-NLDtime} in $L^2(\G,\C^2)$. Let us, then, compute $\imath\f{\partial \Psi}{\partial t}$ using \eqref{eq:duhamel}. Fix $t\in[0,T]$. From the properties of the free propagator one immediately sees that
 \begin{equation}
  \label{eq:der1}
  \imath\f{\partial\,\cU(\,\cdot\,)\Psi_0}{\partial t}=\Dg\,\cU(\,\cdot\,)\Psi_0\qquad\text{in}\quad L^2(\G,\C^2),\quad\forall t>0.
 \end{equation}
 Now, let us define the map $\cN$ as
 \begin{multline}
  \label{eq:non_map}
  \cN(\Psi)[t,x]:=\int^t_0 \cU(t-\tau)\big(|\Psi(\tau)|^{p-2}\Psi(\tau)\big)[x]\dtau\\[.1cm]
  =\int^t_0 \cU(y)\big(|\Psi(t-y)|^{p-2}\Psi(t-y)\big)[x]\dy.
 \end{multline}
 Using that $\Psi\in C^1([0,T],L^2(\G))$, one can compute
 \begin{equation}
 \label{eq:non_der}
 \f{\partial\cN(\Psi)}{\partial t}[t,x]=\cU(t)\big(|\Psi_0|^{p-2}\Psi_0\big)[x]+\int^t_0 \cU(t-\tau)\big(\partial_\tau|\Psi(\tau)|^{p-2}\Psi(\tau)\big)[x]\dtau
\end{equation}
and, arguing as before, obtain that also $\cN(\Psi)\in C^1([0,T],L^2(\G))$. In addition, one can check that for $h$ small
\begin{multline*}
h^{-1}(\cU(h)-I)\big(\cN(\Psi)[t,\cdot\,]\big)[x]\\[.2cm]
=\frac{1}{h}\int^t_0\cU(t+h-\tau)\big(|\Psi(\tau)|^{p-2}\Psi(\tau)\big)[x]\dtau-\frac{1}{h}\int^t_0\cU(t-\tau)\big(|\Psi(\tau)|^{p-2}\Psi(\tau)\big)[x]\dtau\\[.2cm]
=\frac{\cN(\Psi)[t+h,x]-\cN(\Psi)[t,x]}{h}-\frac{1}{h}\int_t^{t+h}\cU(t+h-\tau)\big(|\Psi(\tau)|^{p-2}\Psi(\tau)\big)[x]\dtau
\end{multline*}
Using the Mean Value Theorem, \eqref{eq:nonlin1} and the properties of $\cU(t)$, if one takes the limit in the $L^2(\G,\C^2)$ as $h\to0$ of the above formula, then there results
\begin{equation}
 \label{eq:diff_rep}
 -\imath\Dg\cN(\Psi):=\lim_{h\to0}h^{-1}(\cU(h)-I)\big(\cN(\Psi)\big)=\f{\partial\cN(\Psi)}{\partial t}-|\Psi\vert^{p-2}\Psi.
\end{equation}
Finally, combining this with \eqref{eq:duhamel} and \eqref{eq:der1}, one gets
\[
 \imath\f{\partial \Psi}{\partial t}=\Dg\,\cU(\,\cdot\,)\Psi_0+\imath\Dg\cN(\Psi)-|\Psi|^{p-2}\Psi,
\]
that is \eqref{eq-NLDtime}, thus concluding the proof.
\end{proof}

The second preliminary step required for the proof of the local well posedness is the following regularity lemma.

\begin{lemma}[Regularity enhancement]
\label{lem:reg}
 Let $\G$ be a noncompact metric graph with a finite number of edges and let $p>2$. For every $T>0$, if $\Psi\in L^\infty([0,T],\dom(\Dg))\cap W^{1,\infty}([0,T],L^2(\G,\C^2))$ is a solution of \eqref{eq:duhamel} in $L^2(\G,\C^2)$ for almost every $t\in[0,T]$, then $\Psi\in C([0,T],\dom(\Dg))\cap C^1([0,T],L^2(\G,\C^2))$.
\end{lemma}

\begin{proof}
 Preliminarily, we note that the function $\cU(\,\cdot\,)\Psi_0$ belongs to $C([0,T],\dom(\Dg))$ $\cap C^1([0,T],L^2(\G,\C^2))$ from the Stone's Theorem. Then, owing to \eqref{eq:duhamel}, it is sufficient to prove that $\cN(\Psi)\in C([0,T],\dom(\Dg))\cap C^1([0,T],L^2(\G,\C^2))$, with $\cN(\Psi)$ defined by \eqref{eq:nonlin1}.
 
 First, it is clear that \eqref{eq:nonlin1} holds also for $L^\infty([0,T],\dom(\Dg))$-functions and this implies $\Psi\in C([0,T],L^2(\G,\C^2))$ and, via \eqref{eq:duhamel}, $\cN(\Psi)\in C([0,T],L^2(\G,\C^2))$. On the other hand, one can check that \eqref{eq:non_der} is valid even if $\Psi$ only belongs to $W^{1,\infty}([0,T],L^2(\G,\C^2))$. As a consequence, using again the properties of $\cU(t)$ and the absolute continuity of the Lebesgue integral, one obtains that actually $\cN(\Psi)\in C^1([0,T],L^2(\G,\C^2))$.
 
 Finally, it is left to show that $\Dg\Psi\in C([0,T],L^2(\G,\C^2))$. To this aim, we can use again the representation of $\Dg$ given by \eqref{eq:diff_rep} and then, since $\Psi\in L^\infty([0,T],\dom(\Dg))\cap C([0,T],L^2(\G,\C^2))$ and $\f{\partial\cN(\Psi)}{\partial t}\in C([0,T],L^2(\G,\C^2))$, the claim is proved.
\end{proof}

Now, we have all the ingredients for the proof of the local well-posedness of \eqref{eq:cauchy2} in the operator domain. We only recall that, for every $z$ in the resolvent set of $\Dg$, the resolvent operator $(\Dg-z)^{-1}$ maps $L^2(\G,\C^2)$ onto $\dom(\Dg)$.

\begin{proof}[Proof of Theorem \ref{thm:lwp}: point (i).]
In view of Lemmas \ref{lem:duhamel} and \ref{lem:reg}, in order to prove point (i) of Theorem \ref{thm:lwp} it suffices to show that \eqref{eq:duhamel} admits a unique solution in 
\be
\label{eq:Xspace} 
X_T:=L^\infty([0,T],\dom(\Dg))\cap W^{1,\infty}([0,T],L^2(\G,\C^2))
\ee
for $T>0$ small enough. In particular, in the following we prove that the map $\M: X_T\longrightarrow X_T$, defined by
\be
\label{eq:map_contr}
\M(\Psi)[t,x]:=\cU(t)\Psi_0[x]+\imath\int^t_0 \cU(t-\tau)\big(|\Psi(\tau)|^{p-2}\Psi(\tau)\big)[x]\dtau
\ee
is a contraction if restricted to 
\be
\label{eq:ball}
 B_R=\{\Psi\in X_T:\Psi(0)=\Psi_0\:\text{and}\:\|\Psi\|_{X_T}\leq R\},
\ee
for some suitable $T,\,R>0$, where $\Psi(0)=\Psi_0$ has to be meant as an equality in $L^2(\G,\C^2)$, $\Psi_0\in\dom(\Dg)$ and
\begin{align}
\label{eq:Xnorm}
\|\Psi\|_{X_T} & :=\esssup_{t\in[0,T]}\|\Dg\Psi(t)\|_{L^2(\G,\C^2)}+\esssup_{t\in[0,T]}\|\partial_t\Psi(t)\|_{L^2(\G,\C^2)}\\[.2cm]
           & \sim \|\Psi\|_{L^\infty([0,T],\dom(\Dg))}+\|\Psi\|_{W^{1,\infty}([0,T],L^2(\G,\C^2))}.\nonumber
\end{align}
Since the map $\M$ can be checked to be well defined from $X_T$ into itself (arguing as in the proof of Lemma \ref{lem:reg}) it is left to prove that
\begin{itemize}
 \item[(a)] the map preserves the set $B_R$,
 \item[(b)] the map is contractive,
\end{itemize}
provided $T$ is small enough.

\smallskip
\emph{Item (a): $\M(B_R)\subset B_R$.} Fix $R>0,\,T>0$ and $t\in[0,T]$. Preliminarily, note that, arguing again as in the proof of Lemma \ref{lem:reg} one can easily see that $\M(\Psi)[0,\,\cdot\,]=\psi_0$. In addition, as $\Psi_0\in\dom(\Dg)$, using \eqref{eq:der1}, the commutation of $\Dg$ and $\cU(t)$ on $\dom(\Dg)$, and the unitarity of $\cU(t)$ on $L^2(\G,\C^2)$, there results
\be
\label{eq:propnorm}
\|\cU(\,\cdot\,)\Psi_0\|_{X_T}\leq2\|\Psi_0\|_{\Dg}.
\ee
Furthermore, in order to estimate the integral term of \eqref{eq:map_contr}, we argue as follows. Exploiting the regularity of $\Psi\in B_R$, (again) the propeties of the propagator $\cU(t)$ and the fact that $\Dg^{-1}$ maps $L^2(\G,\C^2)$ onto $\dom(\Dg)$, we can integrate by parts in \eqref{eq:non_map} thus obtaining
\begin{align*}
\imath\cN(\Psi)[t,x]= & \, \imath\int_0^t\cU(t-\tau)\Dg\Dg^{-1}\big(|\Psi(\tau)|^{p-2}\Psi(\tau)\big)[x]\dtau\\[.2cm]
                    = & \, -\int_0^t\f{\partial}{\partial t}\,\cU(t-\tau)\Dg^{-1}\big(|\Psi(\tau)|^{p-2}\Psi(\tau)\big)[x]\dtau\\[.2cm]
                    = & \, \Dg^{-1}\big(|\Psi(t)|^{p-2}\Psi(t)\big)[x]-\cU(t)\Dg^{-1}\big(|\Psi_0|^{p-2}\Psi_0\big)[x]+\\[.2cm]
                      & \, -\int_0^t\cU(t-\tau)\Dg^{-1}\big(\partial_\tau|\Psi(\tau)|^{p-2}\Psi(\tau)\big)[x]\dtau\\[.2cm]
                    = & \, \Dg^{-1}\big(|\Psi(t)|^{p-2}\Psi(t)\big)[x]-\cU(t)\Dg^{-1}\big(|\Psi_0|^{p-2}\Psi_0\big)[x]+\\[.2cm]
                      & \, -\int_0^t\cU(t-\tau)\Dg^{-1}\big(|\Psi(\tau)|^{p-2}\partial_\tau\Psi(\tau)\big)[x]\dtau+\\[.2cm]
                      & \, -(p-2)\int_0^t\cU(t-\tau)\Dg^{-1}\big(|\Psi(\tau)|^{p-4}\Psi(\tau)\re\{\langle\partial_\tau\Psi(\tau),\Psi(\tau)\rangle_{\C^2}\}\big)[x]\dtau.
\end{align*}
Now, easy computations yield
\begin{align*}
\imath\f{\partial}{\partial t}\cN(\Psi)[t,x]= & \, -\imath\cU(t)\big(|\Psi_0|^{p-2}\Psi_0\big)[x]-\imath\int_0^t\cU(t-\tau)\big(|\Psi(\tau)|^{p-2}\partial_\tau\Psi(\tau)\big)[x]\dtau+\\[.2cm]
                                              & \, -\imath(p-2)\int_0^t\cU(t-\tau)\big(|\Psi(\tau)|^{p-4}\Psi(\tau)\re\{\langle\partial_\tau\Psi(\tau),\Psi(\tau)\rangle_{\C^2}\}\big)[x]\dtau\\[.2cm]
                                            =: & \, A_1(t,x)+A_2(t,x)+A_3(t,x)
\end{align*}
and
\[
 \imath\Dg\big(\cN(\Psi)[t,\cdot\,]\big)[x]= \underbrace{|\Psi(t,x)|^{p-2}\Psi(t,x)}_{=:A_4(t,x)}+A_1(t,x)+A_2(t,x)+A_3(t,x),
\]
and hence, in order to conclude, it suffices to estimate $A_1,\dots,A_4$ in $L^\infty([0,T],L^2(\G,\C^2))$. Arguing as in \eqref{eq:propnorm}, one immediately sees that
\[
 \|A_1\|_{L^\infty([0,T],L^2(\G,\C^2))}\leq\|\Psi_0\|_{\Dg}^{p-1}.
\]
On the other hand
\begin{align*}
 \|A_2(t)\|_{L^2(\G,\C^2)}^2+\|A_3(t)\|_{L^2(\G,\C^2)}^2 & \leq C_p\int_0^t\big\||\Psi(\tau)|                ^{p-2}|\partial_\tau\Psi(\tau)|\big\|_{L^2(\G,\C^2)}^2\dtau\\[.2cm]
                                                         & \leq C_p\int_0^t\|\Psi(\tau)\|_{\Dg}^{2p-4}\|\partial_\tau\Psi(\tau)\|_{L^2(\G,\C^2)}^2\dtau,
\end{align*}
so that
\[
 \|A_2\|_{L^\infty([0,T],L^2(\G,\C^2))}+\|A_3\|_{L^\infty([0,T],L^2(\G,\C^2))}\leq C_p T \|\Psi\|_{X_T}^{p-1}.
\]
Moreover, the previous computations also entail that $|\Psi|^{p-2}\Psi\in W^{1,1}([0,T],L^2(\G,\C^2))$ and thus
\[
|\Psi(t,x)|^{p-2}\Psi(t,x)=|\Psi_0(x)|^{p-2}\Psi_0(x)+\int^t_0\partial_\tau|\Psi(\tau,x)|^{p-2}\Psi(\tau,x)\dtau.
\]
Therefore, arguing as before one finds that
\[
 \|A_4\|_{L^\infty([0,T],L^2(\G,\C^2))}\leq \|\Psi_0\|_{\Dg}^{p-1}+C_p T \|\Psi\|_{X_T}^{p-1}
\]
and then, summing up,
\[
\|\M(\Psi)\|_{X_T}\leq 2\|\Psi_0\|_{\Dg}+3\|\Psi_0\|^{p-1}_{\Dg}+C_pT\|\Psi\|^{p-1}_{X_T}.
\]
As a consequence, letting
\[
R:=2\,(2\|\Psi_0\|_{\Dg}+3\|\Psi_0\|^{p-1}_{\Dg}),
\]
when $T$ is sufficiently small we have that $\|\M(\Psi)\|_{X_T}\leq R$, which proves the claim.

\smallskip
\emph{Item (b): $\M$ is contractive on $B_R$.} Following a classical strategy by Kato \cite{K}, we aim at proving that $\M$ is contractive on $B_R$, with $B_R$ endowed with a weaker norm than that of $X_T$. This allows to avoid regularity requirements on the nonlinearity, so that no further prescription on $p$ is required.
\smallskip

Precisely, we consider the following norm
\be
\label{eq:wnorm}
\|\Psi\|_{B_R}:=\|\Psi\|_{L^\infty([0,T],L^2(\G,\C^2))}.
\ee
Hence, using the $C^1$ regularity of the nonlinearity (given by $p>2$) and the unitarity of $\cU(t)$ on $L^2(\G,\C^2)$, we get
\begin{align}
\label{eq:contraction}
\|\M(\Psi)-\M(\Omega)\|_{B_R}\leq & \, T\big\||\Psi|^{p-2}\Psi-|\Omega|^{p-2}\Omega \big\|_{B_R}\nonumber\\[.2cm]
                                 \leq & \, CT\big(\|\Psi\|^{p-2}_{X_T}+\|\Omega\|^{p-2}_{X_T}\big)\|\Psi-\Omega\|_{B_R}\nonumber\\[.2cm]
                                 \leq & 2R^{p-2}CT\|\Psi-\Omega\|_{B_R},\qquad\forall\,\Psi,\Omega\in B_{R}
\end{align}
and then $\M$ is a contraction on $B_R$ endowed with the norm \eqref{eq:wnorm}, for $T>0$ sufficiently small. However, in order to use the Banach-Caccioppoli Theorem one has to also check that $(B_R, \|\,\cdot\,\|_{B_R})$ is a complete metric space.

To this aim, consider a sequence $(\Psi_n)_{n}\subset B_R$ which is a Cauchy sequence with respect to the norm $\|\,\cdot\,\|_{B_R}$. Clearly $\Psi_n\to\Psi$ in $L^\infty([0,T],L^2(\G,\C^2)$. Then, it is left to prove that $\Psi\in B_R$. First, note that by the Riesz representation theorem we have the following (anti-)isomorphism of Banach spaces
\[
X_T\simeq \big(L^1([0,T],\dom(\Dg))+W^{1,1}([0,T],L^2(\G,C^2))\big)^*,
\]
and thus, since $(\Psi_n)_{n}$ is equibounded in $X_T$ by assumption, from the Banach-Alaoglu Theorem $\Psi_n$ converges (up to subsequences) to $\Psi$ in the weak-$*$ topology of $X_T$. In addition, since $\Psi_n\subset B_R$, by the weak-$*$ lower semicontinuity of the norm $\|\Psi\|_{X_T}\leq R$. Finally, if $(\Psi_n)_{n}$ is equibounded in $X_T$, then it is also equibounded in $C([0,T],L^2(\G,\C^2))$. Hence, $\big(\Psi_n(0)\big)_n$ is equibounded in $L^2(\G,\C^2)$ and, since $\Psi_n(0)=\Psi_0$ by assumption, exploiting again Banach-Alaoglu there results that $\Psi(0)=\Psi_0$, which concludes the proof.
\end{proof}

The argument to prove the so-called \emph{blow-up alternative}, i.e. the fact that a solution either is global or it blows up in a finite time, is quite classical. We report it for the sake of completeness.

\begin{proof}[Proof of Theorem \ref{thm:lwp}: point (ii).]
 Let $T_{max}$ be the supremum of the times for which point (i) of Theorem \ref{thm:lwp} holds and let
 \begin{equation}
  \label{eq:bua}
  M_{max}:=\sup_{t\in[0,T_{max})}\|\Psi(t)\|_{\Dg}<\infty.
 \end{equation}
 By definition, there exists $(t_n)_n\subset\R^+$, $t_n\to T_{max}$, such that
 \[
  \lim_n\|\Psi(t_n)\|_{\Dg}=M\leq M_{max}.
 \]
 
 Assume, now, that $T_{max}<+\infty$. First, recall that the proof of point (i) does not depend on the initial time $t_0$, but only on the fact that $\Psi(t_0)\in\dom(\Dg)$. In particular, one sees that, in view of \eqref{eq:bua}, $\|\Psi(t_0)\|_{\Dg}\leq M_{max}$ for every $t_0\in(0,T_{max})$. Hence, the existence time $\tau(t_0)$, that one obtains starting from any $t_0\in(0,T_{max})$ and arguing as in the proof of point (i) of Theorem \ref{thm:lwp}, has to satisfy
 \[
  \tau(t_0)\geq C(M_{max})>0,
 \]
 for some suitable constant $C(M_{max})$ depending only on $M_{max}$. Therefore, if one sets $t_0=t_{\hat{n}}$ with $t_{\hat{n}}>T_{max}-C(M_{max})$, then obtains that the solution of \eqref{eq:cauchy2} does exist beyond $T_{max}$, which is a contradiction.
\end{proof}


\subsection{Mass and energy conservation}

Finally, we prove the conservation of the $L^2$--norm, also called \emph{mass}, and of the energy associated with the solution of \eqref{eq:cauchy2}.

Both the proofs strongly rely on the fact that the solution $\Psi\in C([0,T],\dom(\Dg))$ $\cap C^1([0,T],L^2(\G,\C^2))$ for every $T<T_{max}$; mainly, on the fact that
\begin{equation}
 \label{eq:dom_t}
 \Psi(t)\in\dom(\Dg),\qquad\forall\,t\in[0,T_{max}).
\end{equation}

\begin{remark}
 In this section we use $\f{\partial}{\partial x_e}$ in place of $'$ to denote the derivative with respect to $x_e$ to avoid misunderstandings with the derivative with respect to $t$.
\end{remark}

\begin{proof}[Proof of Theorem \ref{thm:lwp}: point (iii).]
 Let $\Psi=(\Phi,X)^T$ be a solution of \eqref{eq:cauchy2} on $[0,T_{max})$. One can rewrite \eqref{eq-NLDtime} componentwise thus obtaining that, on each edge $e\in\E$,
 \be
 \label{eq:comptime}
 \left\{\begin{array}{l}
  \displaystyle \imath\f{\partial\Phi_e}{\partial t}=-\imath\f{\partial X_e}{\partial x_e}+m\Phi_e-|\Psi_e|^{p-2}\Phi_e\\[.3cm]
  \displaystyle \imath\f{\partial X_e}{\partial t}=-\imath\f{\partial \Phi_e}{\partial x_e}-mX_e-|\Psi_e|^{p-2}X_e
 \end{array}\right.
\ee
Multiplying by $\Phi_e^*$ the former equation and by $X_e^*$ the latter, summing up and taking the imaginary part of the resulting equation one obtains
\[
 \re\bigg\{\Phi_e^*\f{\partial\Phi_e}{\partial t}+X_e^*\f{\partial X_e}{\partial t}\bigg\}=-\re\bigg\{\Phi_e^*\f{\partial X_e}{\partial x_e}+X_e^*\f{\partial\Phi_e}{\partial x_e}\bigg\},\qquad\forall\,e\in\E.
\]
Finally, integrating on $\G$ and using \eqref{eq:dom_t} yield
\begin{align*}
 \f{d\|\Psi(t)\|_{L^2(\G,\C^2)}^2}{dt}= & \, -2\re\bigg\{\sum_{e\in\E}\int_{I_e}\bigg(\Phi_e^*(t,x_e)\f{\partial X_e(t,x_e)}{\partial x_e}\dxe+X_e^*(t,x_e)\f{\partial \Phi_e(t,x_e)}{\partial x_e}\bigg)\dxe\bigg\}\\[.2cm]
 = & \, -2\re\bigg\{\sum_{e\in\E}\int_{I_e}\f{\partial}{\partial x_e}\big(\Phi_e^*(t,x_e)X_e(t,x_e)\big)\dxe\bigg\}\\[.2cm]
 = & \, -2\re\bigg\{\sum_{e\in\E}\big[X_e(t,\ell_e)\Phi_e^*(t,\ell_e)-X_e(t,0)\Phi_e^*(t,0)\big]\bigg\}\\[.2cm]
 = & \, 2\re\bigg\{\sum_{\v\in\K}\Phi_e^*(t,\v)\sum_{e\succ\v}X_e^{\pm}(t,\v)\bigg\}=0,
\end{align*}
which concludes the proof
\end{proof}

\begin{proof}[Proof of Theorem \ref{thm:lwp}: point (iv).]
 First, note that,since $\Dg$ is self-adgoint on $\dom(\Dg)$, the kinetic part of the energy defined by \eqref{eq:energy}, namely
 \[
  \Q(\psi):=\f{1}{2}\int_\G\langle\psi,\Dg\psi\rangle\dx,
 \]
 is the diagonal of a hermitian sesquilinear form, which is continuous with respect to the graph norm. Therefore, from \eqref{eq:solution} and \eqref{eq-NLDtime}, we have that
 \begin{align*}
  \f{d E(t)}{dt}= & \, \re\bigg\{\int_\G\big\langle\partial_t\Psi(t,x),\Dg\Psi(t,x)\big\rangle\dx-\int_\G|\Psi(t,x)|^{p-2}\big\langle\partial_t\Psi(t,x),\Psi(t,x)\big\rangle\dx\bigg\}\\[.2cm]
  = & \, \re\bigg\{\int_\G\big\langle\partial_t\Psi(t,x),\Dg\Psi(t,x)-|\Psi(t,x)|^{p-2}\Psi(t,x)\big\rangle\dx\bigg\}\\[.2cm]
  = & \, \re\bigg\{\int_\G\big\langle\partial_t\Psi(t,x),\imath\partial_t\Psi(t,x)\big\rangle\dx\bigg\}=\re\big\{\imath\|\Psi(t)\|_{L^2(\G,\C^2)}^2\big\}=0,
 \end{align*}
 which concludes the proof.
\end{proof}


\section{Proof of Theorem \ref{thm:standingwaves}}
\label{sec:standingwaves}

In this section, we prove the existence of branches of bound states for the NLDE bifurcating from the trivial solution at the positive threshold of the spectrum of $\Dg$, i.e. $\om=m$ (recall that also here we set $c=1$ for the sake of simplicity). Precisely, we prove how to construct these branches, up to a proper scaling, from the bound states of the NLSE.

Unfortunately, the method holds only for infinite $N$-star graphs. Hence, in the sequel of the section we always assume that $\G$ consists of a finite bunch of half-lines, $\h_1,\dots,\h_N$, all incident at the same vertex $\v$.


\subsection{Rescaling the stationary NLDE}
\label{sec-scaling}

Recall that (setting $\psi=(\phi,\chi)^T$) the equation in \eqref{eq:stat} can be rewritten componentwise as
\begin{equation}
\left\{
\label{eq:componenti}
\begin{array}{l}
\displaystyle-\imath\chi_e'+(m-\omega)\phi_e=\big(|\phi_e|^2+|\chi_e|^2\big)^{(p-2)/2}\phi_e,\\[.3cm]
\displaystyle-\imath\phi_e'-(m+\omega)\chi_e=\big(|\phi_e|^2+|\chi_e|^2\big)^{(p-2)/2}\chi_e,
\end{array}
\right.\qquad e=1,\dots,N.
\end{equation}
The first step of the proof is, then, to show that solving \eqref{eq:componenti} in $\dom(\Dg)$ is equivalent, up to the scaling \eqref{eq:sol_scaling}, to solving \eqref{eq-systaux} in $\dom(\Dg)$.

Since \eqref{eq:sol_scaling} clearly yields that, if $(u,v)^T\in\dom(\Dg)$, then $\psi=(\phi,\chi)^T\in\dom(\Dg)$ (and viceversa), it is sufficient to prove the equivalence between \eqref{eq:componenti} and \eqref{eq-systaux}.

To this aim, fix a generic half-line $e$ and set
\be
\label{eq:rescaling}
\phi_e(x_e)=\alpha u_e(\lambda x_e),\qquad \chi_e(x_e)=\beta v_e(\lambda x_e),\qquad x_e>0,
\ee
where $\alpha,\beta,\lambda>0$ are three constants that must not depend on the halfline $e$ and that will be chosen later. Plugging \eqref{eq:rescaling} into \eqref{eq:componenti}, dividing the first equation by $\lambda\beta$ and the second one by $\beta$ and setting $y_e=\lambda x_e$, there results
\[
 \left\{
 \begin{array}{l}
  \displaystyle -\imath v_e'+\frac{\alpha(m-\omega)}{\lambda\beta}u_e=\f{\alpha^{p-1}}{\lambda\beta}\bigg(|u_e|^2+\frac{\beta^2}{\alpha^2}|v_e|^2\bigg)^{(p-2)/2}u_e\\[.6cm]
  \displaystyle -\imath \frac{\lambda\alpha}{\beta}u_e'-(m+\omega)v_e=\alpha^{p-2}\bigg(|u_e|^2+\frac{\beta^2}{\alpha^2}| v|^2\bigg)^{(p-2)/2}v_e
 \end{array}
 \right.,
\]
where here $'$ denotes the derivative with respect to $y_e$. Now, assuming $\om<m$ and setting
\be
\label{eq:parameters}
\ep:=m-\omega,\qquad\alpha=\frac{\lambda\beta}{\ep},\qquad \alpha^{p-1}=\lambda\beta,\qquad \lambda\alpha=\beta, 
\ee
one gets
\be
\label{eq:diracscaled}
 \left\{
 \begin{array}{l}
  \displaystyle -\imath v_e'+u_e=\big(|u_e|^2+\ep|v_e|^2\big)^{(p-2)/2}u_e\\[.3cm]
  \displaystyle -\imath u_e'-2mv_e+\ep v_e=\ep\big(|u_e|^2+\ep|v_e|^2)^{(p-2)/2}v_e
 \end{array}
 \right.
\ee
(where we also used the fact that $m+\omega=2m-\ep$), which is clearly equal to \eqref{eq-systaux}, but parametrized by $\ep$ in place of $\om$. Note that, while \eqref{eq:diracscaled} is meaningful for any $\ep\in\R$, the equivalence with \eqref{eq:componenti} is valid only for $\ep>0$.

\begin{remark}
 In the new parameter $\ep$, the branch point of the solutions is given by $\ep=0$. Moreover, from \eqref{eq:parameters}, one immediately sees that
 \begin{equation}
 \label{eq:par2}
  \lambda=\sqrt{\ep},\qquad\alpha=\ep^{1/(p-2)},\qquad\beta=\ep^{p/(2p-4)}.
 \end{equation}
\end{remark}

\begin{remark}
\label{rem-restriction1}
We emphasize that the previous computations holds only for infinite $N$-star graphs, since they are the unique example of scale invariant metric graph. As a consequence the equivalence proved in this section is the reason for which we must restrict the topology of the graphs for the proof of Theorem \ref{thm:standingwaves}.
\end{remark}


\subsection{Solutions of the rescaled problem}

In view of the previous section, in order to prove Theorem \ref{thm:standingwaves} it is sufficient to prove that there exists a vector $(u_\ep,v_\ep)^T\in\dom(\Dg)$, $u_\ep,\,v_\ep\neq0$, which solves \eqref{eq:diracscaled} on each half-line (at least for small $\ep$).

To this aim it is convenient to rewrite the problem as follows. First, define the map
\[
\F:\R\times X\times Y\longrightarrow L^2(\G)\times L^2(\G)=:L^2(\G,\C^2),
\]
with
\[
X=\{u\in H^{1}(\G):u\text{ satisfies \eqref{eq-cont_bis}}\}\qquad\text{and}\qquad Y=\{v\in H^{1}(\G):v\text{ satisfies \eqref{eq-kirch_bis}}\},
\]
whose action is given by
\be
\label{eq:map}
\F(\ep,u,v)_{|_e}=\F_{e}(\ep,u_e,v_e):=\begin{pmatrix}-\imath v_e'+u_e-\big(|u_e|^2+\ep|v_e|^2\big)^{(p-2)/2}u_e\\[.2cm]
-\imath u_e'-2mv_e+\ep v_e-\ep\big(|u_e|^2+\ep|v_e|^2\big)^{(p-2)/2}v_e\end{pmatrix}
\ee
for every $e=1,\dots N$. Therefore, solving the original issue is equivalent to solving
\be
\left\{
\begin{array}{l}
\label{eq:funceq}
\displaystyle (\ep,u_\ep,v_\ep)\in\R\times X\times Y\\[.2cm]
\displaystyle u_\ep,\,v_\ep\neq0\\[.2cm]
\displaystyle \F(\ep,u_\ep,v_\ep)=0.
\end{array}
\right.
\ee

\subsubsection{Solutions for $\ep=0$}
\label{sec-nls}

The first step is the investigation of the case $\ep=0$. 

We look for solutions of \eqref{eq:funceq} with $\ep=0$, belonging to $X\backslash\{0\}\times Y\backslash\{0\}$
\[
 \left\{
 \begin{array}{l}
  \displaystyle -\imath v_e'+u_e=|u_e|^{p-2}u_e\\[.3cm]
  \displaystyle -\imath u_e'-2mv_e=0
 \end{array}\right.,\qquad e=1,\dots N\,,
\]
namely, $u$ solves \eqref{eq-NLSaux} and $v_e=\frac{-\imath u_e'}{2m}$. In \cite{ACFN-EPL} it is proved that, if $N$ is odd, then there exists a unique positive solution $U_0$ to \eqref{eq-NLSaux} given by
\begin{equation}
 \label{eq-Nodd}
 U_{0,e}(x_e)=\varphi(x_e),\qquad e=1,\dots,N,
\end{equation}
while, if $N$ is even, then there exists a family $(U_\alpha)_{\alpha\geq0}$ of positive solutions given by
\begin{equation}
\label{eq-Neven}
 U_{\alpha,e}(x_e)=
 \left\{
 \begin{array}{ll}
  \displaystyle \varphi(x_e-\alpha),\qquad \text{if }e=1,\dots,N/2;\\[.2cm]
  \displaystyle \varphi(x_e+\alpha),\qquad \text{if }e=N/2+1,\dots,N,
 \end{array}
 \right.
\end{equation}
where
\begin{equation}
 \label{eq-soliton}
 \varphi(t):=c_p\sech^{\gamma_p}(\delta_{p,m}t),
\end{equation}
and
\[
 c_p:=\bigg(\f{p}{2}\bigg)^{\f{1}{p-2}},\qquad\gamma_p:=\f{2}{p-2},\qquad\delta_{p,m}:=\f{\sqrt{2m}}{\gamma_p}.
\]
Summing up, denoting by $U$ a generic real-valued solution of \eqref{eq-NLSaux}, $(u_0,v_0)^T$ with
\begin{equation}
 \label{eq:uU}
 u_{0,e}=U_e,\qquad v_{0,e}=\frac{-\imath U_e'}{2m},\qquad e=1,\dots,N,
\end{equation}
belongs to $X\backslash\{0\}\times Y\backslash\{0\}$ and satisfies $\F(0,u_0,v_0)=0$.

\begin{remark}
\label{rem-pos}
 Note that real-valued solutions of \eqref{eq-NLSaux} cannot change sign. This follows from the fact that the solution on each half-line must be the restriction of a solution on the real line and it is known that, up to translation and phase multiplication, solutions of the NLSE on the real line are positive and radially decreasing (see, e.g., \cite[Thm. 8.1.6]{CAZ}).
 \end{remark}

\begin{remark}
\label{rmk:derivdefocusing}
By the same scaling argument one could consider the case $-m<\omega<0$ in the regime $\omega\to-m$ (taking again $c=1$). In this case we choose as small parameter 
\[
\varepsilon:=m+\omega\,.
\]
Then, plugging (as before) \eqref{eq:rescaling} in \eqref{eq:componenti} but dividing the first equation by $\alpha$ and the second one by $\lambda\alpha$, there results that the roles of $\alpha,\beta$ and of $u,v$ are exchanged compared to the previous case. More precisely, one finds
\[\displaystyle
\lambda=\sqrt{\varepsilon}\,,\qquad \alpha=\varepsilon^{\frac{p}{2(p-2)}}\,,\qquad \beta=\varepsilon^{\frac{1}{p-2}}\,,
\]
and the scaled system becomes
\[
 \left\{
 \begin{array}{l}
  \displaystyle -\imath v_e'+(2m-\ep )u_e=\varepsilon\big(\ep|u_e|^2+|v_e|^2\big)^{(p-2)/2}u_e\\[.3cm]
  \displaystyle -\imath u_e'-v_e=\big(\ep|u_e|^2+|v_e|^2)^{(p-2)/2}v_e.
 \end{array}
 \right.
\]
 As a consequence, the limit problem at $\ep =0$ is \eqref{eq:defocusing}, in place of \eqref{eq-NLSaux}, and since (as mentioned in Remark \ref{rmk:negativeedge}) this has no nontrivial solution in $\dom(-\Delta_\G)$ the regime $\omega\to-m$ cannot be discussed with our method.

\end{remark}

\subsubsection{Solutions for small $\ep$}
\label{sec-smallep}

The solutions obtained for $\ep=0$ in the previous section suggest the use of the Implicit Function Theorem in order to prove existence of solutions of \eqref{eq:funceq} for small values of $\ep$.

It is not hard to check that the map $\F$ is of class $C^1$, so that one has to prove the following

\begin{prop}
\label{prop:linearized}
Let $U$ be a generic real-valued solution of \eqref{eq-NLSaux}. Then, the differential of $\F$ with respect to the second and third variables evaluated at $(0,u_0,v_0)$, with $u_0,\,v_0$ defined by \eqref{eq:uU}, is an isomorphism from $X\times Y$ to $L^2(\G,\C^2)$.
\end{prop}

We denote the differential of $\F$ with respect to the second and third variables evaluated at $(0,u_0,v_0)$ by $D_{(X, Y)}\F(0,u_0,v_0)$. The first step is to show that

\begin{lemma}
\label{lem-inj}
The operator $D_{(X, Y)}\F(0,u_0,v_0):X\times Y\to L^2(\G,\C^2)$ is injective.
\end{lemma}

\begin{proof}
Our aim is to prove that $\ker\big\{D_{(X, Y)}\F(0,u_0,u_0)\big\}$ is trivial. Suppose preliminarily that $U$ is positive, as this is not restrictive because of Remark \ref{rem-pos}. Easy computations entail that this is equivalent to prove that the unique pair $(h,k)^T\in X\times Y$ that solves
\be
\label{eq:linearnls}
 \left\{\begin{array}{l}
  \displaystyle -\imath k_e'+h_e=(p-1)U_e^{p-2}h_e\\[.2cm]
  \displaystyle -\imath h_e'-2mk_e=0
 \end{array}
 \right.,\qquad e=1,\dots,N,
\ee
is $(0,0)^T$, or rather (arguing as in Section \ref{sec-nls}) that the unique solution $h\in\dom(-\lap)$ of 
\be
\label{eq:linearization}
\left\{
\begin{array}{l}
\displaystyle h\in\dom(-\lap),\\[.2cm]
\displaystyle  \f{h_e''}{2m}+(p-1)U_e^{p-2}h_e=h_e,\qquad e=1,\dots,N
\end{array}
\right.
 \ee
is the trivial one.

To this aim, fix a generic half-line $e$, and discuss all the possible nontrivial solutions of the equation in \eqref{eq:linearization}. More precisely, we first deal with ODE in \eqref{eq:linearization} and then consider vertex conditions.
\smallskip

Simply differentiating the $e$-th equation in \eqref{eq-NLSaux}, one easily sees that $U_{e}'$ is a solution of the equation $e$-th equation in \eqref{eq:linearization} and $U_{e}'\in H^2(\h_e)$. Now, denote by $w$ another solution of the equation, so that $w$ and $U_{e}'$ are linearly independent. Abel's identity implies that the Wronskian determinant of $U'_e$ and $w$ is constant, that is
\be
\label{eq:wronsk}
U''_e w-U'_{e}w'=c,\qquad\mbox{with $c\in\C$.}
\ee 
Clearly, $c\neq0$, since otherwise $w$ and $U_{\hat{e}}'$ would be linearly dependent. Also, integrating (by parts) \eqref{eq:wronsk} we have
\begin{equation}
\label{eq:wronint}
cx_e=U'_e(x_e)w(x_e)-U'_e(0)w(0)-2\int_0^{x_e} U'_e(y_e)w'(y_e)\,dy_e,\qquad\forall x_e\in\h_e.
\end{equation}
As a consequence, if we assume that $w\in H^{2}(\h_e)$ and pass to the limit in \eqref{eq:wronint} as $x_e\to+\infty$, then there results that the first term at the r.h.s. tends to zero, the second and the third ones are finite, while the l.h.s. diverges, which is a contradiction. Hence, the unique (up to a multiplicative constant) solution of the equation in \eqref{eq:linearization} that belongs to $H^{2}(\h_e)$ is given by $U_e'$.

Now, it is left to discuss if the function $U'=(U_e')_{e=1}^N$ is in $\dom(-\lap)$, namely if it satisfies \eqref{eq-cont} and \eqref{eq-kirch}. Note that $U$ is given by \eqref{eq-Nodd} or \eqref{eq-Neven} depending on the fact that $N$ is odd or even (respectively). Assume, first, that is of type \eqref{eq-Nodd}, namely, it consists of the same function \eqref{eq-soliton} on each half-line. In this case \eqref{eq-cont} is satisfied, but an easy computation shows that $U_e$ is strictly concave at $x_e=0$ on each half-line and hence $U'$ cannot satisfy \eqref{eq-kirch}. On the other hand, assume that $U$ is of type \eqref{eq-Neven}. However, since U is not symmetric, but satisfies \eqref{eq-kirch}, $U'$ cannot fulfill \eqref{eq-cont}.

Summing up, there are no nontrivial solutions to \eqref{eq:linearization}, and this concludes the proof.
\end{proof}

Now, we want to prove that $D_{(X, Y)}\F(0,u_0,v_0)$ is surjective, using an argument based on the Fredholm's Alternative. Precisely, from classical Perturbation Theory for linear operators (see e.g. \cite{K}), it suffices to show that $D_{(X, Y)}\F(0,u_0,v_0)$ is the sum of an isomorphism and of a compact operator.

Therefore, let
\begin{equation}
\label{eq-dec}
 D_{(X, Y)}\F(0,u_0,v_0)=J+K(U),
\end{equation}
with $J,K(U):X\times Y\to L^2(\G,\C^2)$ defined by
\be
\label{eq:J}
J[h,k]_{|e}=J_e[h_e,k_e]:=(-\imath k_e'+h_e,-\imath h_e'-2mk_e)^T,\qquad e=1,\dots,N
\ee
and
\be
\label{eq:K}
K(U)[h,k]_{|_e}=K_e(U_e)[h_e,k_e]=(-(p-1)|U_e|^{p-2}h_e,0)^T,\qquad e=1,\dots,N.
\ee

\begin{lemma}
\label{lem-J}
 The operator $J:X\times Y\to L^2(\G,\C^2)$ is an isomorphism.
\end{lemma}

\begin{proof}
 Since $J$ is clearly continuous, one has to prove only injectivity and surjectivity.
 
 \smallskip
 \emph{Step (i): J is injective.} Let $(h,k)^T\in X\times Y$ be a solution of $J[h,k]=0$. Arguing as in the proof of Lemma \ref{lem-inj}, one can check that this is equivalent to claim that
\[
h\in\dom(-\lap)\qquad\text{and}\qquad-\f{h_e''}{2m}+h_e=0,\qquad e=1,\dots,N.
\]
However, multiplying by $h_e^*$, integrating by parts and using \eqref{eq-cont} and \eqref{eq-kirch}
\[
\int_\G\big(|h'|^2+|h|^2\big)\dx=0,
\]
and thus $h\equiv0$ and, as $k_e=\frac{-\imath h_e'}{2m}$ on each halfline, $k\equiv0$.

 \smallskip
 \emph{Step (ii): J is surjective.} In order to prove this we prove that the range of $J$ is dense and closed in $L^2(\G,\C^2)$. Therefore, let $a,\,b\in \bigoplus_{e=1}^n C_0^\infty(\h_e\backslash\{0\})$. We want to prove that there exists $(h,k)^T\in X\times Y$ such that
\be
\label{eq:surj}
J[h,k]=(a,b)^T,
\ee
namely, such that
\[
 \left\{\begin{array}{l}
  \displaystyle k_e'=-\imath h_e +\imath a_e\\[.2cm]
  \displaystyle h_e'=2m\imath k_e+\imath b_e
 \end{array}\right.,\qquad e=1,\dots,N.
\]
First, if we assume that $b\equiv0$, then the problem is to find
\be
\label{eq:bzero2}
 h_1\in\dom(-\lap)\qquad\text{such that}\qquad -\f{h_{1,e}''}{2m}+h_{1,e}=a_e,\qquad e=1,\dots,N,
\ee
or rather, in a weaker form, $h_1\in X$ such that
\[
 q(h_1,\varphi)=\langle a,\varphi\rangle_{L^2(\G)},\qquad\forall\varphi\in X,
\]
with $q:X\times X\to\C$ defined as
\[
  q(\mu,\varphi):=\frac{1}{2m}\sum_{e=1}^N\langle \mu'_e,\varphi'_e\rangle_{L^2(\h_e)}+\langle \mu_e,\varphi_e\rangle_{L^2(\h_e)}.
\]
However, simply using the Lax-Milgram Lemma, one can immediately see that such function $h_1$ does exist and also satisfies
\[
\|h_1\|_{H^1(\G)}\leq\|a\|_{L^2(\G)}.
\]
As a consequence, letting $k_{1,e}=\frac{-\imath h_{1,e}'}{2m}$, the pair $(h_1,k_1)^T\in X\times Y$ solves \eqref{eq:surj} with $b\equiv0$ and fulfills
\be
\label{eq:bounda}
\|(h_1,k_1)^T\|_{H^1(\G,\C^2)}\leq C\|a\|_{L^2(\G)}.
\ee
Now, arguing exactly as before (with $k$ in place of $h$ and $Y$ in place of $X$), one can also find a solution $(h_2,k_2)^T\in X\times Y$ of \eqref{eq:surj} with $a\equiv0$ such that
\[
\|(h_2,k_2)^T\|_{H^1(\G,\C^2)}\leq C\|b\|_{L^2(\G)}.
\]
Hence, using the linearity of $J$, one has that for every $a,\,b\in\bigoplus_{e=1}^n C_0^\infty(\h_e\backslash\{0\})$ there exists $(h,k)^T\in X\times Y$ that satisfies \eqref{eq:surj} and
\be
\label{eq:bound}
\|(h,k)^T\|_{H^1(\G,\C^2)}\leq C\|(a,b)^T\|_{L^2(\G,\C^2)},
\ee
and, since $\bigoplus_{e=1}^nC_0^\infty(\h_e\backslash\{0\})$ is dense in $L^2(\G)$, there results that $\operatorname{Ran}(J)$ is dense in $L^2(\G,\C^2)$.

It is, then, left to prove that $\operatorname{Ran}(J)$ is closed in $L^2(\G,\C^2)$. To this aim, consider a generic sequence
\be
\label{eq:converging}
\big((a_n,b_n)^T\big)_n\subset L^{2}(\G,\C^2),\qquad (a_n,b_n)\to(a,b)\quad\text{in}\:L^{2}(\G,\C^2),
\ee 
such that, for every $n$, there exists $(h_n,k_n)\in X\times Y$ satisfying
\be\label{eq:range}
J[h_n,k_n]=(a_n,b_n).
\ee
Since in Step (i) we proved that $J$ is injective, it admits a left inverse $J^{-1}$, which is a continuous map from $L^2(\G,\C^2)$ to $H^1(\G,\C^2)$ by \eqref{eq:bound} and the density of $\operatorname{Ran}(J)$. Then, by \eqref{eq:converging} and \eqref{eq:range}, we have that
\begin{equation}
\label{eq:closure}
(h_n,k_n)^T=J^{-1}[a_n,b_n]\longrightarrow J^{-1}[a,b]=:(h,k)^T\qquad\text{in}\: H^{1}(\G,\C^2).
\end{equation}
Hence, in order to conclude we just need to show that $(h,k)^T\in X\times Y$, namely, that $h$ satisfies \eqref{eq-cont} and $k$ satisfies \eqref{eq-kirch}. However, this is immediate by \eqref{eq:closure} since $(h_n,k_n)^T\in X\times Y$, for every $n$, and $X,\,Y\hookrightarrow\bigoplus_{e=1}^nC(\h_e)$, and thus $\operatorname{Ran}(J)$ is closed in $L^2(\G,\C^2)$, which completes the proof.
\end{proof}

\begin{lemma}
\label{lem-K}
 The operator $K(U):X\times Y\to L^2(\G,\C^2)$ is compact.
\end{lemma}

\begin{proof}
Preliminarily, assume that $U$ is positive (since it is not restrictive). Since $K(U)$ is clearly continuous, we only prove that, for every bounded sequence $\big((h_n,k_n)^T\big)_{n}\subset X\times Y$, there exist $\big((h_{n_j},k_{n_j})^T\big)_{j}\subset X\times Y$ such that $\big(K(U)[h_{n_j},k_{n-j}]\big)_j$ is convergent in $L^2(\G,\C^2)$.

First note that, up to subsequences,
\be
\label{eq:tail}
(h_n,k_n)\longrightarrow (h,k)\qquad \text{in}\quad L_{loc}^2(\G,\C^2).
\ee
On the other hand, from \eqref{eq-Nodd}, \eqref{eq-Neven}, \eqref{eq-soliton}, one sees that for all $\eta>0$ there exists $M_\eta>0$ such that $U_e(x_e)<\eta$ for every $|x_e|>M_\eta$ and every $e=1,\dots,N$. Thus,
\begin{multline*}
\|K(U)[h_n,k_n]-K(U)[h,k]\|_{L^2(\G,\C^2)}=\|U^{p-2}(h_n-h)\|_{L^2(\Omega^c_\eta,\C^2)}+\\[.2cm]
+\|U^{p-2}(h_n-h)\|_{L^2(\Omega_\eta,\C^2)}\leq C\eta^{p-2}+o(1),\qquad\text{as}\quad n\to\infty.
\end{multline*}
where $\Omega_\eta$ is the compact metric space obtained cutting all the half-lines at distance $M_\eta$ from the common vertex and $\Omega^c_\eta=\G\backslash\Omega_\eta$. Hence
\[
\lim_{n}\|K(U)[h_n,k_n]-K(U)[h,k]\|_{L^2(\G,\C^2)}\leq C\eta^{p-2},\qquad\forall \eta>0,
\]
so that the statement follows.
\end{proof}

Now, we can combine all the previous results to prove Proposition \ref{prop:linearized}.

\begin{proof}[Proof of Proposition \ref{prop:linearized}]
 Since $D_{(X, Y)}\F(0,u_0,v_0)$ is clearly continuous, it is only necessary to prove that it is bijective. From Lemma \ref{lem-inj} we have that it is injective. On the other hand, Lemmas \ref{lem-J} and \ref{lem-K}, in view of \eqref{eq-dec}, show that 
 \[
  D_{(X, Y)}\F(0,u_0,v_0)=:J+K
 \]
 is the sum of an isomorphism and of a compact operator. As a consequence, by the Fredholm's alternative, injectivity and surjectivity are equivalent and then the claim is proved.  To see this, observe that 
 \[
 D_{(X, Y)}\F(0,u_0,v_0)=J+K=J(I+J^{-1}K)\,.
 \]
 Since $J^{-1}$ is an isomorphism, injectivity of $D_{(X, Y)}\F(0,u_0,v_0)$ implies the injectivity of $(I+J^{-1}K)$, which is of the form identity plus compact. Then the claim follows by the standard Fredholm's alternative (see, e.g., \cite[Thm. 6.6]{B}).
 
\end{proof}

Finally, we have all the ingredients to complete the proof of Theorem \ref{thm:standingwaves}.

\begin{proof}[Proof of Theorem \ref{thm:standingwaves}]
 From Section \ref{sec-nls} we know that \eqref{eq:funceq} admits the solution $(0,u_0,v_0)$ with $u_0,\,v_0$ defined by \eqref{eq:uU}, with $U$ a generic real-valued solution of \eqref{eq-NLSaux}. In addition, from Proposition \ref{prop:linearized}, the assumptions of the Implicit Function Theorem are satisfied in $(0,u_0,v_0)$. Hence, there exists $\ep_0>0$ and a $C^1$-map 
\[
(-\ep_0,\ep_0)\ni\ep\mapsto (u_\ep,v_\ep)\in X\times Y
\]
such that $(\ep,u_\ep,v_\ep)$ is a solution of \eqref{eq:funceq} for every $\ep\in(-\ep_0,\ep_0)$. Thus, the proof follows from the results obtained in Section \ref{sec-scaling}.
\end{proof}


\appendix


\section{Nonrelativistc limit of $\Dg$}
\label{sec:nonrel}

We prove  Proposition \ref{prop:nonrel}, namely we show that the operator $\Dg$ (introduced by Definition \ref{defi-dirac}) converges in the norm-resolvent sense to the Schr\"odinger operator either with Kirchhoff or with homogeneous $\delta'$ conditions at the vertices, according to the sign of the renormalization term $\pm mc^2$ (which is known as the ``rest energy'' of the particle and has to disappear when relativistic effects become negligible).

To this aim, we need a representation formula for the resolvent operator of $\Dgc\pm mc^2$ on $\C\backslash\R$. Note that here we will always make explicit the dependence of $\Dgc$ on $c$.

\begin{lemma}
\label{lem-rep}
For every $k\in\C\backslash\R$, there results that
\begin{multline}
\label{eq:resdecomp}
(\Dgc\mp mc^2-k)^{-1}=\\[.2cm]
\bigg(P^{\pm}\pm\frac{\Dm+k}{2mc^2}\bigg)\bigg(\I\mp\left(\frac{-\Delta}{2m}-k\right)^{-1}\frac{k^2}{2mc^2}\bigg)^{-1}\left(\frac{-\Delta}{2m}-k\right)^{-1},
\end{multline}
where $\Dm$ is an operator with the same domain as $\Dg$ and action
\[
 \Dm\psi_{|_e}=\Dme\psi_e:=-\imath c\sigma_{1}\psi_e',\qquad\forall e\in\E,
\]
$P^{\pm}$ are the projections on the first and the second component of the spinors given by
\[
 \begin{pmatrix} 1 & 0 \\ 0 & 0 \end{pmatrix}\qquad\text{and}\qquad\begin{pmatrix} 0 & 0 \\ 0 & 1 \end{pmatrix}\,,
\]
respectively, and $-\Delta$ is as in \eqref{eq-biglap}.
\end{lemma}

\begin{remark}
 Note that, both the operator $\Dm$ and the operator $-\Delta$ are clearly self-adjoint by definition.
\end{remark}

\begin{proof}[Proof of Lemma \ref{lem-rep}]
We only consider the case $(\Dgc-mc^2-k)^{-1}$, since the proof for $(\Dgc+mc^2-k)^{-1}$ is completely analogous.

Preliminarily note that
\be
\label{eq:dsquared}
\Dgc^2=-c^2\Delta+m^2c^4\I,
\ee
namely, $\Dgc^2$ has action
\[
 \Dgc^2\psi_{|e}=\Dec^2\psi_e=-c^2\psi_e''+m^2c^4\psi_e
\]
and the same domain as $-\Delta$, i.e.
\[
\dom(\Dgc^2):=\{\psi=(\varphi,\chi)^T\in H^2(\G,\C^2):\varphi\:\text{satisfies \eqref{eq-cont}$\&$\eqref{eq-kirch} and}\:\chi\:\text{satisfies \eqref{eq-contd}$\&$\eqref{eq-kirchd}}\},
\]
and it is, therefore, clearly self-adjoint.

Now, for every $k\in\C\backslash\R$, one easily sees that
\[
(\Dgc-mc^2-k)(\Dgc+mc^2+k)=-c^2\Delta-2mc^2k-k^2,
\]
whence, with some algebra,
\be
\label{eq:resexpr}
(\Dgc-mc^2-k)^{-1}=\frac{\Dgc+mc^2+k}{2mc^2}\left(\frac{-\Delta}{2m}-k-\frac{k^2}{2mc^2}\right)^{-1}.
\ee
In addition, using the identity 
\[
(A+B)^{-1}=(\I+A^{-1}B)^{-1}A^{-1}
\]
with 
\[
A:=\frac{-\Delta}{2m}-k\qquad\text{and}\qquad B=-\frac{k^2}{2mc^2},
\]
there results
\be
\label{eq:resexpress2}
(\Dgc-mc^2-k)^{-1}=\frac{\Dgc+mc^2+k}{2mc^2}\bigg(\I-\left(\frac{-\Delta}{2m}-k\right)^{-1}\frac{k^2}{2mc^2}\bigg)^{-1}\left(\frac{-\Delta}{2m}-k\right)^{-1}.
\ee
Then, as one can easily check
\[
\frac{\Dgc+mc^2+k}{2mc^2}=P^++\frac{\Dm+k}{2mc^2}
\]
which, plugging into \eqref{eq:resexpress2}, yields \eqref{eq:resdecomp}.
\end{proof}

\begin{proof}[Proof of Proposition \ref{prop:nonrel}]
 Again, we only consider the case $(\Dgc-mc^2-k)^{-1}$, since the proof for $(\Dgc+mc^2-k)^{-1}$ is completely analogous.
 
The goal is to prove that 
\[
 (\Dgc-mc^2-k)^{-1}\longrightarrow P^+\left(\frac{-\Delta}{2m}-k\right)^{-1},\qquad\text{in}\quad\L(L^2(\G,\C^2)),\quad  \text{as}\quad c\to\infty.
\]
By \eqref{eq:resdecomp}, easy computations yield
\[
 (\Dgc-mc^2-k)^{-1}-P^+\left(\frac{-\Delta}{2m}-k\right)^{-1}=A_c+B_c,
\]
where
\[
 A_c:=P^+(H_c-\I)\la\qquad\text{and}\qquad B_c:=\ga_c H_c\la,
\]
with
\[
 \la:=\left(\frac{-\Delta}{2m}-k\right)^{-1},\qquad H_c:=\bigg(\I-\la\tf{k^2}{2mc^2}\bigg)^{-1},\qquad\ga_c:=\f{\Dm+k}{2mc^2}.
\]

First, one sees that $\la$, $\tf{k^2}{2mc^2}$ and (when $c$ is large) $H_c$ belong to $\L(L^2(\G,\C^2))$ and that
\[
 \big\|\tf{k^2}{2mc^2}\big\|_{\L(L^2(\G,\C^2))}=\tf{|k|^2}{2mc^2}
\]
Hence, one gets  $A_c\to0$ in $\L(L^2(\G,\C^2))$, as $c\to\infty$.

On the other hand, the Open Mapping Theorem gives
\[
 \ga_c\in\L(\dom(\Dgc),L^2(\G,\C^2))\qquad\text{and}\qquad H_c\in \L(\dom(-\Delta)).
\]
Note that a priori such operator is bounded only on $L^2(\G,\C^2)$. However, for large $c$, its restriction to $\dom (-\Delta)$ can be expanded in Neumann series 
\[
H_c=\sum_{n\geq0}\left( \la\tf{k^2}{2mc^2}\right)^n \,,
\]
thus proving $ H_c\in \L(\dom(-\Delta))$. Moreover,
\[
 \dom(-\Delta)\hookrightarrow \dom(\Dgc)
\]
and 
\[
 \|\ga_c\|_{\L(\dom(\Dgc),L^2(\G,\C^2))}\lesssim c^{-1}
\]
Therefore, $B_c\to0$ in $\L(L^2(\G,\C^2))$, as $c\to\infty$, which concludes the proof.
\end{proof}


\end{document}